\documentclass{article}
\usepackage{}

\usepackage{amssymb}
\usepackage{latexsym}
\usepackage{extarrows}
\usepackage[dvips]{graphicx}
\usepackage{graphicx,color}
\usepackage{hyperref}

\let\cal=\mathcal
\newtheorem{theo}{Theorem}[section]

\newtheorem{lemma}[theo]{Lemma}
\newtheorem{claim}[theo]{Claim}
\newtheorem{coro}[theo]{Corollary}

\newtheorem{proposition}[theo]{Proposition}

\newtheorem{result}[theo]{Result}

\newcommand*{\qed}{\hfill\ensuremath{\square}}

\begin{document}

\title{Exact extremal non-trivial cross-intersecting families\thanks{{E-mail:  $^\dag$wu@hunnu.edu.cn (B. Wu), $^\ddag$huajunzhang@usx.edu.cn (H. Zhang, corresponding author)}}}

\author{
Biao Wu$^{\dag}$,
Huajun Zhang  $^{\ddag}$\\[2ex]
 {\small $^{\dag}$ School of Mathematics and Statistics, Hunan Normal University} \\
{\small Changsha, Hunan, 410081, P.R. China } \\
{\small $^{\ddag}$ Department of Mathematics, Shaoxing University} \\
{\small Shaoxing, Zhejiang, 321004, P.R. China }}

\maketitle

\begin{abstract}
Two families $\mathcal{A}$ and $\mathcal{B}$ of sets are called cross-intersecting if each pair of sets $A\in \mathcal{A}$ and  $B\in \mathcal{B}$ has nonempty intersection. Let $\cal{A}$ and ${\cal B}$ be two cross-intersecting families of $k$-subsets and $\ell$-subsets of $[n]$. Matsumoto and Tokushige [J. Combin. Theory Ser. A 52 (1989) 90--97] studied the extremal problem of the size $|\cal{A}||\cal{B}|$ and obtained the uniqueness of extremal families whenever $n\ge 2 \ell\ge 2k$, building on the work of Pyber. This paper will explore the second extremal size of $|\cal{A}||\cal{B}|$ and obtain that if $\mathcal{A}$ and $\mathcal{B}$ are not the subfamilies of Matsumoto--Tokushige's extremal families, then, for $n\ge  2\ell >2k$ or $n>  2\ell=2k$,
\begin{itemize}
\item[1)]either $|\cal{A}||\cal{B}|\le \left({\binom{n-1}{k-1}}+{\binom{n-2 }{k-1}}\right){\binom{n-2}{\ell-2}}$ with the unique extremal families (up to isomorphism)
    \[\mbox{$\mathcal{A}=\{A\in {\binom{[n]}{k}}: 1\in A \: \rm{ or} \: 2\in A\}$ \quad and \quad  $\mathcal{B}=\{B\in {\binom{[n]}{\ell}}: [2] \subseteq B\}$};\]
\item[2)] or $|\cal{A}||\cal{B}|\le \left({\binom{n-1}{k-1}}+1\right)\left({\binom{n-1}{\ell-1}}-{\binom{n-k-1}{\ell-1}}\right)$ with the unique extremal families (up to isomorphism)
    \[\mbox{$\mathcal{A}=\{A\in {\binom{[n]}{k}}: 1\in A\}\cup \{[2,k+1] \}$\quad and \quad $\mathcal{B}=\{B\in {\binom{[n]}{\ell}}: 1\in B, B\cap [2,k+1]\neq \emptyset \}$.}\]
\end{itemize}
The bound ``$n\ge  2\ell >2k$ or $n>  2\ell=2k$" is sharp for $n$.
To achieve the above results, we establish some size-sensitive inequalities for cross-intersecting families. As by-products, we will recover the main results of Frankl and Kupavskii [European J. Combin. 62 (2017) 263--271].

\end{abstract}

Keywords: nontrivial cross-intersecting, extremal combinatorics.

\section{Introduction}

For any positive integer $n$, let $[n]=\{1,2,\ldots ,n\}$.
The power set $2^{[n]}$ is the collection of all subsets of $[n]$. Denote by $\binom{[n]}{k}$ the collection
of all subsets of $[n]$ of size $k$. A family $\mathcal{A}\subseteq 2^{[n]}$ is called {\it intersecting}  if  $A\cap B\neq \emptyset$ for all $A,B\in \mathcal{A}$. Two families $\mathcal{A},\mathcal{B}\subseteq 2^{[n]}$ are called {\it cross-intersecting} if $A\cap B\neq \emptyset$ for all $A\in \mathcal{A}$ and $B\in \mathcal{B}$. Notice that an intersecting family $\cal A$ coincides with the cross-intersecting families $\cal A$ and $\cal A$. The study on extremal problems of intersecting families dates back to the celebrated work of Erd\H{o}s-Ko-Rado.

\begin{theo}[Erd\H{o}s--Ko--Rado \cite{EKR1961}]\label{EKR1961}
Let $n$ and $k$ be two positive integers with $n\ge 2k$. If a family $\mathcal{F} \subseteq \binom{[n]}{k}$ is intersecting,
then
 $ |\mathcal{F}| \le {\binom{n-1}{k-1}}$.
\end{theo}

The Erd\H{o}s--Ko--Rado theorem has been investigated in depth with various proofs;
see \cite{Kat1972,Fra2021} for the famous Katona cycle proof,
\cite{FHW2006} for  proof applying the method of linearly independent polynomials,
and \cite{Day1974,Kat1964,FF2012} for short proofs  using
the shadow technique. Readers can refer to \cite{FT2016,FT2018} for the
comprehensive surveys on recent results related to EKR theorem. It is known to us that in 1986 Pyber gave the first generalization of the EKR theorem from an intersecting family to two cross-intersecting families in the way of size product.

\begin{theo} [Pyber \cite{Pyber}] \label{Pyber}
Let $n, k$ and $\ell$ be three positive integers. If two families $\mathcal{A}\subseteq {\binom{[n]}{k}}$ and $\mathcal{B}\subseteq {\binom{[n]}{\ell}}$ are cross-intersecting, then
\begin{itemize}
\item[\rm {1)}] $|\mathcal{A}||\mathcal{B}|\leq {\binom{n-1}{k-1}}^2$ if $k=\ell$ and $n\ge 2k$, and
\item[{\rm 2)}] $ |\mathcal{A}||\mathcal{B}|\leq {\binom{n-1}{k-1}}{\binom{n-1}{\ell-1}}$ if $k>\ell$ and $n\ge 2k+\ell-2$.
\end{itemize}
\end{theo}

In Pyber's second case $\ell>k$,  the upper bound was given for $n\ge 2\ell+k-2$, which was improved to
$n\ge 2\ell > 2k$ by Matsumoto and Tokushige in 1989. Together with Pyber's first case, it completes the extremal problem of size product for cross-intersecting families $\mathcal{A}\subseteq {\binom{[n]}{k}}$ and $\mathcal{B}\subseteq {\binom{[n]}{\ell}}$.

\begin{theo}[Matsumoto--Tokushige \cite{MT1989}] \label{theoMT}
Let $n, k$ and $\ell$ be three positive integers  with $n\ge 2\ell \ge 2k$. If two families $\mathcal{A}\subseteq {\binom{[n]}{k}}$ and $\mathcal{B}\subseteq {\binom{[n]}{\ell}}$ are cross-intersecting, then
$ |\mathcal{A}||\mathcal{B}|\leq {\binom{n-1}{k-1}}{\binom{n-1}{\ell-1}}$.
Equality holds if and only if
$\mathcal{A}=\{A\in {\binom{[n]}{k}}: x\in A\}$ and $\mathcal{B}=\{B\in {\binom{[n]}{\ell}}: x\in B\}$ for some $x\in [n]$.
\end{theo}

In Pyber's first case $\ell=k$, a nontrivial extremal problem is to find the upper bound of the product $|\cal A| |\cal B|$ in the case $\cap_{F\in\cal A\cup\mathcal{B}} F= \emptyset$. In 2017, Frankl and Kupavskii obtained a sharp upper bound under the conditions: $|\mathcal{A}|\leq{\binom{n-1}{k-1}}\leq|\mathcal{B}|$ and $\cap_{F\in\mathcal{B}} F= \emptyset$. More recently, the first author and Xiong completed this nontrivial extremal problem after investigating Frankl-Kupavskii's exceptional cases and established the unique extremal families.

\begin{theo}[Frankl--Kupavskii \cite{FK2017}, Wu--Xiong \cite{WX}] \label{theoFKWX}
Let $n$ and $k$ be two positive intgers with $ n\ge  2k+1$.
If two families $\mathcal{A}\subseteq {\binom{[n]}{k}}$ and $\mathcal{B}\subseteq {\binom{[n]}{\ell}}$ are cross-intersecting with $\cap_{F\in\mathcal{A}\cup\cal B}F= \emptyset$, then
\begin{equation*}\label{equwx}
 |\mathcal{A}||\mathcal{B}|\leq \left({\binom{n-1}{k-1}}-{\binom{n-k-1}{k-1}}\right)\left({\binom{n-1}{k-1}}+1\right).
 \end{equation*}
Equality holds if and only if $\mathcal{A}=\{A\in {\binom{[n]}{k}}: x\in A, A\cap D\neq \emptyset\}$ and $\mathcal{B}=\{B\in {\binom{[n]}{k}}: x\in B\}\cup \{D\}$ for some $x\in [n]$ and $D\in {\binom{[n]\setminus \{x\}}{k}}$.
\end{theo}

Motivated by Theorem \ref{theoMT} and \ref{theoFKWX}, it is natural to consider the extremal problem of size product for nontrivial cross-intersecting families in Pyber's second case $\ell>k$, which is exactly our main result of this paper.

\begin{theo} [main result]\label{maintheo}
Let $n, k, \ell$  be positive integers with $n\ge  2\ell >2k$ or $n>  2\ell=2k$.
If two families $\mathcal{A}\subseteq {\binom{[n]}{k}}$ and $\mathcal{B}\subseteq {\binom{[n]}{\ell}}$ are cross-intersecting with $\cap_{F\in \mathcal{A} \cup \mathcal{B}} F=\emptyset$, then
\begin{itemize}
\item[\rm {1)}]either $|\cal{A}||\cal{B}|\le \left({\binom{n-1}{k-1}}+{\binom{n-2 }{k-1}}\right){\binom{n-2}{\ell-2}}$ with the unique extremal families (up to isomorphism)
    \[\mbox{$\mathcal{A}=\{A\in {\binom{[n]}{k}}: 1\in A \: \rm{ or} \: 2\in A\}$ \quad and \quad  $\mathcal{B}=\{B\in {\binom{[n]}{\ell}}: [2] \subseteq B\}$};\]
\item[\rm {2)}] or $|\cal{A}||\cal{B}|\le \left({\binom{n-1}{k-1}}+1\right)\left({\binom{n-1}{\ell-1}}-{\binom{n-k-1}{\ell-1}}\right)$ with the unique extremal families (up to isomorphism)
    \[\mbox{$\mathcal{A}=\{A\in {\binom{[n]}{k}}: 1\in A\}\cup \{[2,k+1] \}$\quad and \quad $\mathcal{B}=\{B\in {\binom{[n]}{\ell}}: 1\in B, B\cap [2,k+1]\neq \emptyset \}$.}\]
\end{itemize}
 \end{theo}
 In the above theorem, the bound ``$n\ge  2\ell >2k$ or $n>  2\ell=2k$" is sharp for $n$, which can be seen from the following example.
\medskip
\newline
{\bf Example.} Case 1. $\ell>k$. Let $n=2\ell-1$ and $k=\ell-1$. Setting
$\mathcal{A}\subseteq {\binom{[n]}{k}}$ with $|\mathcal{A}|=\lfloor\frac 1 2 {\binom{2\ell-1}{\ell-1}}\rfloor$ and $\mathcal{B}=\{B\in {\binom{[n]}{\ell}}: [n]\setminus B \notin \mathcal{A}\}$,
thus $|\mathcal{B}|=\lceil\frac 1 2 {\binom{2\ell-1}{\ell}}\rceil$.
Then $\mathcal{A}$ and $\mathcal{B}$ are non-trivial cross-intersecting with $|\mathcal{A}||\mathcal{B}|=\lfloor\frac 1 2 {\binom{2\ell-1}{\ell-1}}\rfloor \lceil\frac 1 2 {\binom{2\ell-1}{\ell}}\rceil$.
As $\lfloor\frac 1 2 {\binom{2\ell-1}{\ell-1}}\rfloor+\lceil\frac 1 2 {\binom{2\ell-1}{\ell}}\rceil={\binom{n-1}{k-1}}+{\binom{n-2 }{k-1}}+\binom{n-2}{\ell-2}={\binom{n-1}{k-1}}+1+{\binom{n-1}{\ell-1}}-{\binom{n-k-1}{\ell-1}}$,
it is easy to see that $|\mathcal{A}||\mathcal{B}|> \max \{ ({\binom{n-1}{k-1}}+{\binom{n-2 }{k-1}}){\binom{n-2}{\ell-2}}, ({\binom{n-1}{k-1}}+1)({\binom{n-1}{\ell-1}}-{\binom{n-k-1}{\ell-1}})\}$.

Case 2. $\ell=k$. Setting $n=2k$ and $\mathcal{A}=\mathcal{B}=\{A\in \binom{[2k]}{k}:1\in A\}\cup \{[2,k+1]\}\setminus \{1,k+2,k+3,...,2k\}$.
Thus $\mathcal{A}$ and $\mathcal{B}$ are non-trivial cross-intersecting  with $|\mathcal{A}||\mathcal{B}|=\binom{n-1}{k-1}^2$.
Similar to Case 1, it is easy to see that $|\mathcal{A}||\mathcal{B}|> \max \{ ({\binom{n-1}{k-1}}+{\binom{n-2 }{k-1}}){\binom{n-2}{\ell-2}}, ({\binom{n-1}{k-1}}+1)({\binom{n-1}{\ell-1}}-{\binom{n-k-1}{\ell-1}})\}$.
\medskip

The paper is organized as follows.
In Section 2, we introduce some tools needed in the proof.
In Section 3, we present some size-sensitive inequalities for cross-intersecting families.
In Section 4, we shall give the complete proof of our main result.

\section{Preliminaries}


Given a family $\mathcal{F}\subseteq {\binom{[n]}{k}}$ with $1\le k\le n$. For any positive integer $\ell$ with $\ell\le k$, the family $\Delta_{\ell}(\mathcal{F}):=\bigcup_{F\in\mathcal{F}}{\binom{F}{\ell}}$ is called
the {\em $\ell$-shadow} of $\cal{F}$. Suppose the family ${\binom{[n]}{k}}$ is equipped with an antilexicographic order $``<"$, i.e.,  $A<B$ in ${\binom{[n]}{k}}$ if and only if $\max (A \setminus B) \cup (B \setminus A) \in B$. Denote by $\mathcal{C}^{(k)}(m)$ the collection of the smallest $m$  elements of ${\binom{[n]}{k}}$. For example, $\mathcal{C}^{(3)}(5)=\left\{\left\{1,2,3\right\},\left\{1,2,4\right\},\left\{1,3,4\right\},
\left\{2,3,4\right\},\left\{1,2,5\right\}\right\}$.
Below is the celebrated Kruskal-Katona Theorem.

\begin{theo}[Kruskal \cite{Kru1963}, Katona \cite{Kat1966}]\label{kkl}
Let $n, k$  be positive integers with $n\ge k$. If $\mathcal{F}\subseteq {\binom{[n]}{k}}$ is a family of size $m$, then for any positive integer $\ell$ with $\ell <k$, we have
\[
|\Delta_{\ell}(\mathcal{F})|\ge |\Delta_{\ell}(\mathcal{C}^{(k)}(m)|.
\]
\end{theo}

\begin{theo}[Lov\'{a}sz \cite{Lov1979}]\label{Lov1979}
Let $n, k$  be positive integers with $n\ge k$. If $\mathcal{F}\subseteq {\binom{[n]}{k}}$ is a family of size $|\mathcal{F}|={\binom{x}{k}}:={\frac{x(x-1)\cdots (x-k+1)}{k!}}$ for some real number $x\ge k$, then for any positive integer $\ell$ with $\ell <k$, we have
\[
|\Delta_{\ell}(\mathcal{F})|\ge {\binom{x}{\ell}}.
\]
\end{theo}
For any positive integers $m,k$ with $m\ge k$, $m$ can be uniquely written as the following form
\[
m={\binom{a_1}{k}}+\binom{a_2}{k-1}+\cdots +{\binom{a_t}{{k-t+1}}},
\]
for positive integers $a_1,a_2,\ldots,a_t$ with $a_1>a_2>\cdots >a_t> k-t$,
which is so-called $k$-{\em cascade} representation of $m$.
Let $k$ be a positive integer.

\begin{theo}[M\"ors, Theorem 7 in \cite{mors}]\label{mors1} With the assumptions as Theorem \ref{kkl}, let $m={\binom{a_1}{k}}+\binom{a_2}{k-1}+\dots+{\binom{a_t}{k-t+1}}$ be the $k$-cascade of $m$,
then the equality $|\Delta_{\ell}(\mathcal{F})|=|\Delta_{\ell}(\mathcal{C}^{(k)}(m)|$ in Theorem \ref{kkl} has a  unique solution $\mathcal{F}=\mathcal{C}^{(k)}(m)$ (up to isomorphism)
 if and only if
\[
\ell\ge t, \quad {\text or~} {\binom{a_1+1}{k}}-1=m,\quad {\text or~} m\le k+1.
\]
\end{theo}

\begin{theo}[M\"ors, Theorem 1.6 in \cite{mors}]\label{mors2}
Let $k,\ell,n$ be positive integers with $n\ge k+\ell$. If two families $\mathcal{A}\subseteq {\binom{[n]}{k}}$ and $\mathcal{B}\subseteq {\binom{[n]}{\ell}}$ are cross-intersecting with $\cap_{F\in\mathcal{A}\cup\cal B}F= \emptyset$, then we have either
  \begin{eqnarray*}
|\mathcal{A}| \le {\binom{n-1}{k-1}}-{\binom {n-\ell-1}{k-1}} + 1  \quad
{\text or}\quad
|\mathcal{B}| \le {\binom{n-1}{\ell-1}}-{\binom{n-k-1}{\ell-1}} + 1.
\end{eqnarray*}
Moreover, both equalities only achieve simultaneously.
\end{theo}

The lexicographic order ``$\prec$" on the set $\binom{[n]}{k}$ is defined to be: $F\prec G$ in $\binom{[n]}{k}$ if and only if $\min (F\setminus G)< \min (G\setminus F)$.
For $1\leq m\leq \binom{n}{k}$, let $\mathcal{L}(n, k, m)$ denote the set of the smallest $m$ elements of $\binom{[n]}{k}$  in the lexicographic order.
For example,
$\mathcal{L}(10,3,\binom{9}{2})=\{A\in\binom{[10]}{3}: 1\in A\}$.
The following lemma is a base tool in our proofs concerning cross-intersecting families.

\begin{lemma} [Hilton \cite{Hil1976}] \label{lem1}
Let $n$, $k$ and $\ell$ be two positive integers.
If two families $\mathcal{A} \subseteq {\binom{[n]}{k}}$
and $\mathcal{B} \subseteq {\binom{[n]}{\ell}}$ are  cross-intersecting,
then  $\mathcal{L}(n, k, |\mathcal{A}|)$
and $\mathcal{L}(n, \ell,  |\mathcal{B}|)$ are cross-intersecting as well.
\end{lemma}
\medskip

To end this section, we present some  preparations, which contain the core idea of the proofs.
\begin{result}\label{r1} Let $a,b,c,d$ be four positive numbers. Then $(c-a)(d+b)\leq cd$ if and only if $\frac{c}{d+b}\leq \frac{a}{b}$.\end{result}

We say two cross-intersecting families $\mathcal A\subseteq \binom{[n]}{k}$ and $\mathcal B\subseteq \binom{[n]}{\ell}$ are {\em maximal} if adding any element of ${\binom{[n]}{k}}\setminus \mathcal{A}$ to $\mathcal{A}$
or  any element of
${\binom{[n]}{\ell}}\setminus \mathcal{B}$  to $\mathcal{B}$, the resulting two families are not cross-intersecting.
Let $\mathcal A\subseteq \binom{[n]}{k}$ and $\mathcal B\subseteq \binom{[n]}{\ell}$ be two given maximal cross-intersecting families.
For any $\mathcal X\subseteq \binom{[n]}{k}\setminus\mathcal A$, set $N_{\mathcal B}(\mathcal X)=\{B\in\mathcal B: \mbox{$B\cap A = \emptyset$ for some $A\in \mathcal X$}\}$. If there is no confusion, we abbreviate $N_{\mathcal B}(\mathcal X)$ as $N(\mathcal X)$. It is easy to find that
$\mathcal A\cup\mathcal X$ and $\mathcal B\setminus N(\mathcal X)$ are two maximal cross-intersecting families as well, and by Result \ref{r1},
\[\mbox{$|\mathcal A\cup\mathcal X||\mathcal B\setminus N_{\mathcal B}(\mathcal X)|<|\mathcal{A}||\mathcal{B}|\Longleftrightarrow
 {\frac{|\mathcal{B}|}{|\mathcal{A}|+|\mathcal{X}|}} < {\frac{|N(\mathcal{X})|}{|\mathcal{X}|}}.$}\]
Now we introduce a result to determine the lower bound of the ratio $\frac{|N(\mathcal{X})|}{|\mathcal{X}|}$. A bipartite graph $G(X,Y)$ is called {\em regular} if all vertices  in the same part are of the same degree.
The following result is well-known.
\begin{result}\label{factbipartite}
Let $G=(X,Y)$ be a regular bipartite graph.
Then for any nonempty subset $A$ of $X$,  we have ${\frac{|N(A)|}{|A|}}\ge {\frac{|Y|}{|X|}}$.
\end{result}

Corresponding to a pair of families $\mathcal F$ and $\mathcal G$, it is so called Kneser graph $KG(\mathcal F, \mathcal G)$ with partite sets $\mathcal F$ and $\mathcal G$, and edge set
$\{\{F, G\} : F\cap G= \emptyset, F\in \mathcal F, G \in \mathcal G\}$.
For $\mathcal Y\subseteq \mathcal G$ and $\mathcal X\subseteq \mathcal F$, denote by $N_{\mathcal F}(\mathcal Y)=N_{KG(\mathcal F, \mathcal G)}(\mathcal Y)$ the neighbour of $\cal{Y}$ in $KG(\mathcal F, \mathcal G)$, and $N_{\mathcal G}(\mathcal X)=N_{KG(\mathcal F, \mathcal G)}(\mathcal X)$.

For integers $r,s,u,n$ with $r,s\le u<n$,
set
\begin{eqnarray*}
\mathcal{F}^{n,u,r,s}= \left\{F \in {\binom{[n]}{u}}: F\cap [r]\neq \emptyset  \:\: {\rm or} \:\: [r+1,r+s] \subseteq F \right\}\end{eqnarray*}
and
\begin{eqnarray*}
\mathcal{F}_{n,u,r,s}=\left\{F\in {\binom{[n]}{u}}: [r]\subseteq F {\rm \:\:  and \:\: } F\cap [r+1,r+s] \neq \emptyset \right\}.
\end{eqnarray*}
We notice that $\mathcal{F}^{n,u,r,s}$ and $\mathcal{F}_{n,u,r,s}$, which satisfy the following properties, play a crucial role in our proofs.
\begin{itemize}
\item[(P1)]$\mathcal F^{n,u,r,s}\subseteq \mathcal F^{n,u,r,s-1}$, $\mathcal F_{n,u,r,s-1}\subseteq\mathcal F_{n,u,r,s}$;
\item[(P2)]$|\mathcal{F}^{n,u,r,s}|=\sum_{1\leq i\leq r }\binom{n-i}{u-1}+\binom{n-r-s}{u-s}  \mbox{\ \ and\ \ }
 |\mathcal{F}_{n,u,r,s}|=\binom{n-r}{u-r}-\binom{n-r-s}{u-r}.$;
\item[(P3)] $\mathcal{F}^{n,u,r,s}$ and $\mathcal F_{n,u,r,s}$ ($\mathcal{F}_{n,u,r,s}$ and $\mathcal{F}^{n,u,r,s}$) are two maximal cross-intersecting families;
 \item[(P4)]$KG(\mathcal F_{n,u,r,s}\backslash\mathcal F_{n,u,r,s-1},\mathcal F^{n,u,r,s-1}\backslash\mathcal F^{n,u,r,s})$ is regular for $s\geq 2$.
\end{itemize}

Suppose $\mathcal F^{n,u,r,s}\subseteq\mathcal B\subseteq \mathcal F^{n,u,r,s-1}$.
Then
  \begin{eqnarray}\label{baseine}
  |\mathcal A||\mathcal B|<|\mathcal{F}_{n,u,r,s}||\mathcal F^{n,u,r,s}|\Leftrightarrow\frac{|\mathcal{F}_{n,u,r,s}|}{|\mathcal B|} < \frac{|\mathcal{F}_{n,u,r,s}\backslash\mathcal A|}{|\mathcal B\backslash\mathcal F^{n,u,r,s}|}=\frac{|N_{\mathcal{F}_{n,u,r,s}}(\mathcal B\backslash\mathcal F^{n,u,r,s})|}{|\mathcal B\backslash\mathcal F^{n,u,r,s}|}=\frac{|\mathcal{F}_{n,u,r,s}|-|\mathcal A|}{|\mathcal B|-|\mathcal F^{n,u,r,s}|}.
  \end{eqnarray}
To up bound the ratio $\frac{|\mathcal{F}_{n,u,r,s}|-|\mathcal A|}{|\mathcal B|-|\mathcal F^{n,u,r,s}|}$, we divide $\mathcal{F}_{n,u,r,s}\backslash\mathcal{F}_{n,u,r,s-1}$ and
 $\mathcal F^{n,u,r,s-1}\backslash\mathcal F^{n,u,r,s}$ into some parts as following: For $1\leq i\leq k-r$, $1\le j\le n$ define
\[\mathcal{F}_{n,k,r,s,i}=\left\{A\in\binom{[n]}{k}: \mbox{$ [r]\subseteq A$  and $A\cap[r+1,r+s+i]=[r+s,r+s+i-1]$}\right\}\]
and
\[\mathcal F^{n,\ell,r,s-1,j}=\left\{B\in\binom{[n]}{\ell}: \mbox{$B\cap [r]=\emptyset$ and $[r+1,r+s+j]\cap B=[r+1,r+s-1]\cup\{r+s+j\}$}\right\}.\]
We notice that $\mathcal{F}_{n,k,r,s,i}$ and $\mathcal F^{n,\ell,r,s-1,j}$ satisfy the following properties.
\begin{itemize}
\item[(P1$'$)]$\mathcal{F}_{n,u,r,s}\backslash\mathcal{F}_{n,u,r,s-1}=\cup_{i=1}^{k-r} \mathcal{F}_{n,k,r,s,i}$ and $\mathcal F^{n,u,r,s-1}\backslash\mathcal F^{n,u,r,s}=\cup_{j=1}^{n} \mathcal{F}^{n,k,r,s-1,i}$;
\item[(P2$'$)]$|\mathcal{F}_{n,k,r,s,i}|=\binom{n-r-s-i}{k-r-i}$ and $|\mathcal F^{n,\ell,r,s-1,i}|=\binom{n-r-s-i}{\ell-s}$;
\item[(P3$'$)]$N_{\mathcal{F}_{n,u,r,s}}(\mathcal F^{n,\ell,r,s-1,1}\cup\cdots\cup\mathcal F^{n,\ell,r,s-1,t})=\mathcal{F}_{n,k,r,s,1}\cup\cdots\cup\mathcal{F}_{n,k,r,s,t}$ for all $1\leq t\leq k$;
\item[(P4$'$)] $KG(\mathcal{F}_{n,k,r,s,t},\mathcal F^{n,\ell,r,s-1,t})$ is regular. 
\end{itemize}

\medskip
To prove our main results in details in the next two sections, we need two more preparations. Firstly,
the following two facts will be frequently used to scale the inequalities in our proofs without any extra explanation.
\begin{itemize}
\item[(F1)]If $a>b>0$, then $\frac b a < \frac {b+1}{a+1}$, that is, $b(a+1)<(b+1)a$;
\item[(F2)]Let $a,b,c,d,x$ be positive numbers. If $\frac b a < x$ and $\frac d c < x$ then $\frac {b+d} {a+c} < x$.
\end{itemize}
Secondly, almost each proof will be started with the following conditions.
\begin{itemize}
\item[(C1)] $\mathcal{A} \subseteq {\binom{[n]}{k}}$
and $\mathcal{B} \subseteq {\binom{[n]}{\ell}}$ are two  maximal cross-intersecting families;
\item[(C2)]
If $|\mathcal{A}|>{\binom{n-1}{k-1}}$ or $|\mathcal{B}|>{\binom{n-1}{\ell-1}}$, then we assume that
$\mathcal A=\mathcal L(n,k,|\mathcal A|)$ and $\mathcal B=\mathcal L(n,\ell,|\mathcal B|)$,
that is, $\mathcal{A}$ and $\mathcal{B}$ are the collections of the first $|\mathcal{A}|$ and $|\mathcal{B}|$ sets
in ${\binom{[n]}{k}}$ and ${\binom{[n]}{\ell}}$ with respect to the lexicographical order, respectively.
\end{itemize}

\section{Some size-sensitive results}
This section will be devoted to some sensitive-size results for two cross-intersecting families, which play an important role in the proof of Theorem \ref{maintheo} and extend Frankl and Kupavskii's size-sensitive theorem \cite{FK2017}.
The method we used is the techniques  with relevant bipartite graphs.

\begin{proposition} \label{prop1}
Let $n, k, \ell$  be positive integers with $n\ge  2\ell \ge 2k>0$.
If two families $\mathcal{A}\subseteq {\binom{[n]}{k}}$ and $\mathcal{B}\subseteq {\binom{[n]}{\ell}}$ are cross-intersecting with
 $ {\binom{n-1}{k-1}}+{\binom{n-s-1}{k-s}} \le |\mathcal{A}| \le {\binom{n-1}{k-1}}+{\binom{n-3}{k-2}}$ for some $3\le s \le k$, then
\begin{equation*}\label{equa}
 |\mathcal{A}||\mathcal{B}|\leq  \left({\binom{n-1}{k-1}}+{\binom{n-s-1}{k-s}}\right) \left({\binom{n-1}{\ell-1}}-{\binom{n-s-1}{\ell-1}}\right).
 \end{equation*}
 \end{proposition}

 \begin{proposition} \label{prop2}
Let $n, k, \ell$  be positive integers with $n\ge  2\ell > 2k>0$.
If two families $\mathcal{A}\subseteq {\binom{[n]}{k}}$ and $\mathcal{B}\subseteq {\binom{[n]}{\ell}}$ are cross-intersecting with
 $ {\binom{n-1}{k-1}}+{\binom{n-3}{k-2}} \le |\mathcal{A}|\le {\binom{n-1}{k-1}}+\sum_{i=3}^4{\binom{n-i}{k-2}}$, then
\begin{equation*}
 |\mathcal{A}||\mathcal{B}|\leq  \left({\binom{n-1}{k-1}}+{\binom{n-3}{k-2}}\right)\left(\binom{n-1}{\ell-1}-\binom{n-3}{\ell-1}\right).
 \end{equation*}
 \end{proposition}

 \begin{proposition} \label{prop3}
Let $n, k, \ell$  be positive integers with $n<  \ell^2$ and $\ell\ge k>0$.
If two families $\mathcal{A}\subseteq {\binom{[n]}{k}}$ and $\mathcal{B}\subseteq {\binom{[n]}{\ell}}$ are cross-intersecting with
$ {\binom{n-1}{k-1}}+\sum_{i=3}^4{\binom{n-i}{k-2}} \le |\mathcal{A}|\le {\binom{n-1}{k-1}}+\sum_{i=3}^5{\binom{n-i}{k-2}}$, then
\begin{equation*}
 |\mathcal{A}||\mathcal{B}|\leq \left({\binom{n-1}{k-1}}+\sum_{i=3}^4{\binom{n-i}{k-2}}\right)\left( \binom{n-2}{\ell-2}+\binom{n-4}{\ell-3} \right).
 \end{equation*}
 \end{proposition}
 Propositions \ref{prop1}--\ref{prop3} imply that $ |\mathcal{A}||\mathcal{B}|\leq  \left({\binom{n-1}{k-1}}+1\right)\left({\binom{n-1}{\ell-1}}-{\binom{n-k-1}{\ell-1}}\right)$ in the appropriate conditions of $n,k,\ell$.
  However, we can not use the similar method to handle the cases  $ {\binom{n-1}{k-1}}+\sum_{i=3}^4{\binom{n-i}{k-2}} \le |\mathcal{A}|\le {\binom{n-1}{k-1}}+\sum_{i=3}^5{\binom{n-i}{k-2}}$ with $n\ge \ell^2$, and $ {\binom{n-1}{k-1}}+\sum_{i=3}^{5}{\binom{n-i}{k-2}} \le |\mathcal{A}|\le {\binom{n-1}{k-1}}+\sum_{i=3}^{\ell+1}{\binom{n-i}{k-2}}$.
  We will handle it in the next section by using a subtle polynomial.

  \begin{proposition} \label{prop4}
Let $n, k, \ell$  be positive integers with $n\ge  2\ell > 2k>0$.
If two families $\mathcal{A}\subseteq {\binom{[n]}{k}}$ and $\mathcal{B}\subseteq {\binom{[n]}{\ell}}$ are cross-intersecting with
$ {\binom{n-1}{k-1}}+\sum_{i=3}^{\ell+1}{\binom{n-i}{k-2}} \le |\mathcal{A}|\le {\binom{n-1}{k-1}}+{\binom{n-2}{k-1}}$, then
\begin{equation*}
 |\mathcal{A}||\mathcal{B}|\leq  \left({\binom{n-1}{k-1}}+{\binom{n-2}{k-1}}\right) {\binom{n-2}{\ell-2}}.
 \end{equation*}
 \end{proposition}

 \begin{proposition} \label{prop5}
Let $n, k, \ell$  be positive integers with $n\ge  2\ell > 2k>0$.
If two families $\mathcal{A}\subseteq {\binom{[n]}{k}}$ and $\mathcal{B}\subseteq {\binom{[n]}{\ell}}$ are cross-intersecting with
 $|\mathcal{A}|\ge \sum_{i=1}^{s}{\binom{n-i}{k-1}}$  for some $s\ge 2$, then
\begin{equation*}
 |\mathcal{A}||\mathcal{B}|\leq  \sum_{i=1}^s{\binom{n-i}{k-1}}\binom {n-s}{\ell-s}.
 \end{equation*}
 \end{proposition}
Propositions \ref{prop4} and \ref{prop5} imply that $|\cal{A}||\cal{B}|\le \left({\binom{n-1}{k-1}}+{\binom{n-2 }{k-1}}\right){\binom{n-2}{\ell-2}}$  in the appropriate conditions of $n,k,\ell$.
The following corollary follows immediately.
\begin{coro}\label{coroA}
Let $n, k, \ell$  be positive integers with $n\ge  2\ell > 2k>0$.
Let $\mathcal{A}\subseteq {\binom{[n]}{k}}$ and $\mathcal{B}\subseteq {\binom{[n]}{\ell}}$ be  two cross-intersecting families.
If the size of $\mathcal{A}$ satisfies one of the following three conditions,
\newline
{\rm (i)} ${\binom{n-1}{k-1}}+1 \le |\mathcal{A}|\le {\binom{n-1}{k-1}}+\sum_{i=3}^4{\binom{n-i}{k-2}}$;
\newline
{\rm (ii)} ${\binom{n-1}{k-1}}+\sum_{i=3}^4{\binom{n-i}{k-2}} \le |\mathcal{A}|\le {\binom{n-1}{k-1}}+\sum_{i=3}^5{\binom{n-i}{k-2}}$ with $n<\ell^2$;
\newline
{\rm (iii)} $|\mathcal{A}|\ge {\binom{n-1}{k-1}}+\sum_{i=3}^{\ell+1}{\binom{n-i}{k-2}}$, then we have
\[
 |\mathcal{A}||\mathcal{B}|\leq  \max \left\{\left({\binom{n-1}{k-1}}+1\right)\left({\binom{n-1}{\ell-1}}-{\binom{n-k-1}{\ell-1}}\right), \left({\binom{n-1}{k-1}}+{\binom{n-2}{k-1}}\right) {\binom{n-2}{\ell-2}} \right\}.
\]
\end{coro}

\begin{theo}\label{theo3}
Let $n, k, \ell$ be positive integers with $\ell \ge k >0$ and $n\ge  k+\ell$.
If two families $\mathcal{A}\subseteq {\binom{[n]}{k}}$ and $\mathcal{B}\subseteq {\binom{[n]}{\ell}}$ are cross-intersecting with
  $|\mathcal B|\ge \sum_{1\leq i\leq r}\binom{n-i}{\ell-1}+\binom{n-r-s}{\ell-s}$ for some $r \ge 1$ and  $2\leq s\leq \ell $,
then
$$|\mathcal{A}| |\mathcal{B}| \le \left(\binom{n-r}{k-r}-\binom{n-r-s}{k-r}\right) \left(\sum_{1\leq i\leq r}\binom{n-i}{\ell-1}+\binom{n-r-s}{\ell-s}\right).$$
\end{theo}

The following corollary follows from Theorem \ref{theo3} immediately.
\begin{coro}\label{ctheo3}
Let $n, k, \ell$ be positive integers with $\ell \ge k >0$ and $n\ge  k+\ell$.
If two families $\mathcal{A}\subseteq {\binom{[n]}{k}}$ and $\mathcal{B}\subseteq {\binom{[n]}{\ell}}$ are cross-intersecting with
 $|\mathcal B| \ge \binom{n-1}{\ell-1}+1$,
then
$$|\mathcal{A}| |\mathcal{B}|\le \left(\binom{n-1}{k-1}-\binom{n-\ell-1}{k-1}\right) \left(\binom{n-1}{\ell-1}+1\right).$$
\end{coro}
In a similar way to Subsection 4.3, the extremal families in above are unique.
Corollary \ref{ctheo3} infers Theorem \ref{theoFKWX}, and the case $k=\ell$ of Theorem \ref{theo3} infers  the following Frankl and Kupavskii's theorem.
\begin{theo}[Frankl--Kupavskii \cite{FK2017}]\label{theoFK}
Let $n, k$ be positive integers with $n>2k>0$.
If two families $\mathcal{A}, \mathcal{B}\subseteq \binom{[n]}{ k}$ are cross-intersecting with
 $|\mathcal{B}|\ge \binom{n-1}{ k-1}+\binom{n-i }{ k-i+1}$ holds for some $3\le i \le k+1$, then
\begin{equation*}\label{FK2}
 |\mathcal{A}||\mathcal{B}|\leq \left(\binom{n-1}{ k-1}+\binom{n-i}{ k-i+1}\right) \left(\binom{n-1}{ k-1}-\binom{n-i }{ k-1}\right).
 \end{equation*}
\end{theo}

\subsection{The proof of Proposition \ref{prop1}--\ref{prop5}}
{\em Proof of Proposition \ref{prop1}}\:\:
It is enough to consider $\cal{A}$ and $\cal{B}$ satisfying (C1) and (C2). Suppose that ${\binom{n-1}{k-1}}+{\binom{n-t-2}{k-t-1}} \le |\mathcal{A}| < {\binom{n-1}{k-1}}+{\binom{n-t-1}{k-t}}$ for some $2\leq t\leq s-1$.
Recall the definitions of $\mathcal{F}^{n,u,r,s}$ and $\mathcal{F}_{n,u,r,s}$ in Section 2.
From the assumptions (C1) and (C2) we have
\[
\mathcal{F}^{n,k,1,t+1} \subseteq \mathcal{A} \subseteq \mathcal{F}^{n,k,1,t}\:\: {\rm and}\:\:
\mathcal{F}_{n,\ell,1,t} \subseteq \mathcal{B} \subseteq \mathcal{F}_{n,\ell,1,t+1}.\]
Clearly, the bipartite  graph $KG(\mathcal{F}^{n,k,1,t} \setminus \mathcal{F}^{n,k,1,t+1},\mathcal{F}_{n,\ell,1,t+1} \setminus \mathcal{F}_{n,\ell,1,t})$ is regular and we have
\[
|\mathcal{F}^{n,k,1,t} \setminus \mathcal{F}^{n,k,1,t+1}|={\binom{n-t-2}{k-t}}\:\: {\rm and}\:\:|\mathcal{F}_{n,\ell,1,t+1} \setminus \mathcal{F}_{n,\ell,1,t}|={\binom{n-t-2}{\ell-2}}.
 \]
 It follows from Result \ref{factbipartite} that  $\frac{|N(\mathcal A\backslash \mathcal{F}^{n,k,1,t})|}{|\mathcal A\backslash\mathcal{F}^{n,k,1,t}|}\ge \frac{|\mathcal{F}_{n,\ell,1,t+1} \setminus \mathcal{F}_{n,\ell,1,t}|}{|\mathcal{F}^{n,k,1,t} \setminus \mathcal{F}^{n,k,1,t+1}|}={\frac{{\binom{n-t-1}{\ell-2}}}{{\binom{n-t-1}{k-t+1}}}}$.
 We may assume that $\cal A \setminus \mathcal{F}_{n,\ell,1,t} \neq \emptyset$, since otherwise $\cal A = \mathcal{F}_{n,\ell,1,t}$ and $\cal B= \mathcal{F}_{n,\ell,1,t+1}$, and hence the equation in Proposition \ref{prop1} holds.
To obtain $|\mathcal{A}||\mathcal{B}|<|\mathcal{F}^{n,k,1,t+1}||\mathcal{F}_{n,\ell,1,t+1}|,$
 by inequality (\ref{baseine}) it suffices to prove that
${\frac{|\mathcal{F}_{n,\ell,1,t+1}|}{|\mathcal{F}^{n,k,1,t+1}|}} < {\frac{{\binom{n-t-1}{\ell-2}}}{{\binom{n-t-1}{k-t+1}}}}.$ We first consider the case $t\ge 4$.
Note that $|\mathcal{F}^{n,k,1,t+1}|>{\binom{n-1}{k-1}}$
and
$|\mathcal{F}_{n,\ell,1,t+1}|<{\binom{n-1}{\ell-1}}$ clearly. Therefore
\begin{eqnarray*}
{\frac{|\mathcal{F}_{n,\ell,1,t+1}|{\binom{n-t-1}{k-t+1}}}{|\mathcal{F}^{n,k,1,t+1}|{\binom{n-t-1}{\ell-2}}}}
<{\frac{{\binom{n-1}{\ell-1}}{\binom{n-t-1}{k-t+1}}}{{\binom{n-1}{k-1}}{\binom{n-t-1}{\ell-2}}}}
= {\frac{(n-k)(n-k-1)(k-1)(k-2)\cdots (k-t+2)}{(n-\ell)(n-\ell-1)\cdots (n-t-\ell+2)(\ell-1)}}\le 1,
\end{eqnarray*}
where the last inequality follows from the following three facts: (1) $(n-k)(k-2)\le (n-\ell)(\ell-1)$; (2) $(n-k-1)(k-1)\le (n-\ell-1)(n-\ell-2)$;
(3) $k-i\le n-\ell-i$ for all $3\le i \le t-2$.

Now assume that $t=3$.
Note that $$|\mathcal{F}_{n,\ell,1,4}|={\binom{n-1}{\ell-1}}-{\binom{n-4}{\ell-1}}={\binom{n-2}{\ell-2}}+{\binom{n-3}{\ell-2}}+{\binom{n-4}{\ell-2}}
=x{\binom{n-3}{\ell-2}},$$
where $x={\frac{n-2}{n-\ell}}+1+{\frac{n-\ell-1}{n-3}}$.  The assumption $n\geq 2\ell$ implies $x\le {\frac{2\ell-2}{2\ell-\ell}}+1+{\frac{\ell-1}{2\ell-3}}=3.5+{\frac{1}{2(2\ell-3)}}-{\frac{2}{\ell}}$, which means
\begin{eqnarray}\label{inex}x\leq \left\{
          \begin{array}{ll}
            3, & \hbox{if $\ell=2$ or $3$;} \\
            3.5, & \hbox{if $\ell\geq 4$.}
          \end{array}
        \right.
\end{eqnarray}
Therefore we have
\begin{eqnarray*}
{\frac{|\mathcal{F}_{n,\ell,1,4}|\binom{n-4}{k-2}}{|\mathcal{F}^{n,k,1,4}|\binom{n-4}{\ell-2}}}
<{\frac{x\binom{n-3}{\ell-2}{\binom{n-4}{k-2}}}{{\binom{n-1}{k-1}}{\binom{n-4}{\ell-2}}}}
= {\frac{x(n-k)(n-k-1)(k-1)}{(n-1)(n-2)(n-\ell-1)}}.
\end{eqnarray*}
Let $f(k)=(n-k)(n-k-1)(k-1)$.
Since ${\frac{f(k+1)}{f(k)}}={\frac{(n-k-1)k}{(n-k)(k-1)}}>1$,
we have $f(k)\le f(\ell)$ and
\begin{eqnarray*}
{\frac{x(n-k)(n-k-1)(k-1)}{(n-1)(n-2)(n-\ell-1)}}\le {\frac{x(n-\ell)(\ell-1)}{(n-1)(n-2)}}.
\end{eqnarray*}
Let $g(n)={\frac{x(n-\ell)(\ell-1)}{(n-1)(n-2)}}$. The fact $n(n-\ell-2)-(n-2)(n-\ell-1)=n-\ell-2\ge 0$ implies  ${\frac{g(n+1)}{g(n)}}={\frac{(n-2)(n-\ell-1)}{n(n-\ell-2)}} \le 1$.
Hence we have $g(n)\le g(2\ell)={\frac{x\ell(\ell-1)}{(2\ell-1)(2\ell-2)}}={\frac{x\ell}{2(2\ell-1)}}\leq 1$ for all $\ell\geq 2$ by inequality (\ref{inex}).
This completes the proof. \qed
\medskip
\newline
{\em Proof of Proposition \ref{prop2}}\:\:
Denote
\begin{small}
\[\mathcal A_1=\left\{A\in\binom{[n]}{k}:  1\in A \:{\rm or}\: \{2,3\}\subseteq A\right\},\:\: \mathcal B_1=\left\{B\in \binom{[n]}{\ell}:\mbox{$1\in B$ and $\{2,3\}\cap B\neq \emptyset$}\right\},\]
\end{small}
\begin{small}\[\mathcal A_2=\left\{A\in\binom{[n]}{k}:  1\in A \:{\rm or}\: \{2,3\}\subseteq A\:{\rm or}\: \{2,4\}\subseteq A\right\},\:\: \mathcal B_2=\left\{B\in \binom{[n]}{\ell}:1\in B \:{\rm and}\: 2 \in B \:{\rm or}\:
\{3,4\}\subseteq B\right\}.\]\end{small}
We know that
$|\mathcal A_1|=\binom {n-1}{k-1}+\binom {n-3}{k-2}$, $|\mathcal A_2|=\binom {n-1}{k-1}+\binom {n-3}{k-2}+\binom {n-4}{k-2}$, $|\mathcal B_1|=\binom {n-2}{\ell-2}+\binom {n-3}{\ell-2}$, and $|\mathcal B_2|=\binom {n-2}{\ell-2}+\binom {n-4}{\ell-3}.$
From the assumptions (C1) and (C2) we have
$\mathcal A_1 \subseteq \mathcal{A} \subseteq \mathcal A_2$ and
$\mathcal B_2 \subseteq \mathcal{B} \subseteq \mathcal B_1.$
Clearly, $KG(\mathcal A_1 \setminus \mathcal A_2,\mathcal B_2 \setminus \mathcal B_1)$ is regular and we have
\[
|\mathcal A_1 \setminus \mathcal A_2|=\binom{n-4}{k-2}\:\: {\rm and}\:\:|\mathcal B_2 \setminus \mathcal B_1|=\binom{n-4}{\ell-2}.
 \]
Similarly,  we may assume that $\cal A \setminus \cal A_1 \neq \emptyset$.
To obtain $|\mathcal{A}||\mathcal{B}|<|\mathcal A_1||\mathcal B_1|,$
 by inequality (\ref{baseine}) it suffices to prove that
$ \frac{|\mathcal B_1|}{|\mathcal A_1|} < \frac{\binom{n-4}{\ell-2}}{\binom{n-4}{k-2}}.$
However,
\[\frac{|\mathcal B_1|}{|\mathcal A_1|}=\frac{\binom {n-2}{\ell-2}+\binom {n-3}{\ell-2}}{\binom {n-1}{k-1}+\binom {n-3}{k-2}}<\frac{\binom{n-2}{\ell-2}+2\binom{n-4}{\ell-2}}{\binom{n-2}{k-1}+2\binom{n-3}{k-2}}<\frac{\binom{n-4}{\ell-2}}{\binom{n-4}{k-2}},\]
where the last inequality follows from $\frac{\binom{n-2}{\ell-2}}{\binom{n-2}{k-1}}=\frac{(k-1)(n-k-1)}{(n-\ell)(n-\ell-1)}\frac{\binom{n-4}{\ell-2}}{\binom{n-4}{k-2}}< \frac{\binom{n-4}{\ell-2}}{\binom{n-4}{k-2}}$ and $\frac{\binom{n-4}{\ell-2}}{\binom{n-3}{k-2}}<\frac{\binom{n-4}{\ell-2}}{\binom{n-4}{k-2}}$.
This completes the proof. \qed
\medskip
\newline
{\em Proof of Proposition \ref{prop3}}\:\:
Denote
\begin{small}\[\mathcal A_3=\left\{A\in\binom{[n]}{k}:  1\in A \:{\rm or}\: \{2,i\}\subseteq A\:{\rm for \: some}\: i\in [3,5]\right\},\:\: \mathcal B_3=\left\{B\in \binom{[n]}{\ell}:1\in B \:{\rm and}\: 2 \in B \:{\rm or}\:
[3,5]\subseteq B\right\}.\]\end{small}
We know that
$|\mathcal A_3|={\binom{n-1}{k-1}}+\sum_{i=3}^5{\binom{n-i}{k-2}}$, $|\mathcal B_3|=\binom {n-2}{\ell-2}+\binom {n-5}{\ell-4}.$
Recall the definitions of $\cal A_2$ and $\cal B_2$ in the proof of Proposition \ref{prop2}.
From the assumptions (C1) and (C2) we have
$\mathcal A_2 \subseteq \mathcal{A} \subseteq \mathcal A_3$ and
$\mathcal B_3 \subseteq \mathcal{B} \subseteq \mathcal B_2.$
Clearly, $KG(\mathcal A_2 \setminus \mathcal A_3,\mathcal B_3 \setminus \mathcal B_2)$ is regular and we have
\[
|\mathcal A_2 \setminus \mathcal A_3|=\binom{n-5}{k-2}\:\: {\rm and}\:\:|\mathcal B_3 \setminus \mathcal B_2|=\binom{n-5}{\ell-3}.
 \]
 Similarly,  we may assume that $\cal A \setminus \cal A_2 \neq \emptyset$.
For  $n \le \ell^2$,
to obtain $|\mathcal{A}||\mathcal{B}|\le |\mathcal A_2||\mathcal B_2|,$
 it suffices to prove that
\begin{eqnarray}\label{iprop3}
\frac{{\binom{n-2}{\ell-2}}+{\binom{n-4}{\ell-3}}}{{\binom{n-1}{k-1}}+\binom{n-3}{k-2}+\binom{n-4}{k-2}}\le \frac{{\binom{n-5}{\ell-3}}}{{\binom{n-5}{k-2}}}.
\end{eqnarray}
We will prove that
$\frac{{\binom{n-4}{\ell-3}}}{\binom{n-3}{k-2}+\binom{n-4}{k-2}}<\frac{{\binom{n-5}{\ell-3}}}{{\binom{n-5}{k-2}}}$ and $ \frac{{\binom{n-2}{\ell-2}}}{{\binom{n-1}{k-1}}}\le \frac{{\binom{n-5}{\ell-3}}}{{\binom{n-5}{k-2}}},$ then inequality (\ref{iprop3}) follows. Firstly,
\[\frac{{\binom{n-4}{\ell-3}}}{\binom{n-3}{k-2}+\binom{n-4}{k-2}}< \frac{{\binom{n-4}{\ell-3}}}{2\binom{n-4}{k-2}} =\frac{n-k-2}{2(n-\ell-1)} \frac{{\binom{n-5}{\ell-3}}}{{\binom{n-5}{k-2}}}< \frac{{\binom{n-5}{\ell-3}}}{{\binom{n-5}{k-2}}}.\]
Secondly, as $\frac{k-1}{\ell-2}<\frac{k}{\ell-1}$,  $\frac{n-k-1}{n-1}<\frac{n-k}{n}$,  $\frac{n-k-2}{n-\ell-1}<\frac{n-k}{n-\ell}$, we obtain that
\begin{small}
\begin{eqnarray*}
\frac{\binom{n-2}{\ell-2}/\binom{n-1}{k-1}}{\binom{n-5}{\ell-3}/\binom{n-5}{k-2}}= \frac{(k-1)(n-k)(n-k-1)(n-k-2)}{(\ell-2)(n-1)(n-\ell)(n-\ell-1)}
<   \frac{k(n-k)^3}{(\ell-1)n(n-\ell)^2}:=\frac{\ell}{\ell-1}\frac{x(y-x)^3}{y(y-1)^2},
\end{eqnarray*}
\end{small}
\newline
where $k=x\ell$ and $n=y\ell$ with $0<x\le 1$ and $2\le y\le \ell$.
For $y\ge 4$, as $\frac{ \partial (x(y-x)^3)}{ \partial x}=(y-x)^2(y-4x)\ge 0$, we have
\[\frac{\ell}{\ell-1}\frac{x(y-x)^3}{y(y-1)^2}\le \frac{\ell}{\ell-1} \frac{y-1}{y} \le 1.\]
For $y\le 4$,
it is easy to see that
\[\frac{\ell}{\ell-1}\frac{x(y-x)^3}{y(y-1)^2}\le \frac{7}{6}\frac{\frac y 4(y-\frac y 4)^3}{y(y-1)^2}\le \frac{7}{6}\frac{\frac 4 4(4-\frac 4 4)^3}{4(4-1)^2}=\frac 7 8<1.\]
This completes the proof. \qed
\medskip
\newline
{\em Proof of Proposition \ref{prop4}}\:\:
As $|\mathcal{A}|\ge {\binom{n-1}{k-1}}+\sum_{i=3}^{\ell+1}{\binom{n-i}{k-2}}$, from the assumptions (C1) and (C2) we have
$\mathcal{A} \subseteq \left\{A\in {\binom{[n]}{k}}: 1\in A, \: {\rm or } \:   \{2,i\} \subseteq A \: {\rm for \: all } \: i\in [3,\ell+1] \right\}, $
and hence $\mathcal{B} \subseteq \left\{B\in {\binom{[n]}{\ell}}: [2] \subseteq B \right\}$.
 Thus $|\cal B|\le {\binom{n-2}{\ell-2}}$, which combines the condition $|\cal A|\le {\binom{n-1}{k-1}}+{\binom{n-2}{k-1}}$, we obtain
  $|\mathcal{A}||\mathcal{B}|\le \left({\binom{n-1}{k-1}}+{\binom{n-2}{k-1}}\right) {\binom{n-2}{\ell-2}}.$
This completes the proof. \qed
\medskip
\newline
{\em Proof of Proposition \ref{prop5}}\:\:
From the assumptions (C1) and (C2), if $|\mathcal{A}|> {\binom{n-1}{k-1}}+{\binom{n-2}{k-1}}+\cdots+{\binom{n-\ell}{k-1}}$  then $\cal B=\emptyset$, and we are done.
So we assume that ${\binom{n-1}{k-1}}+{\binom{n-2}{k-1}}+\cdots+{\binom{n-r}{k-1}}\le |\mathcal A|\le {\binom{n-1}{k-1}}+{\binom{n-2}{k-1}}+\cdots+{\binom{n-r-1}{k-1}}$ for some $2\le r\le \ell$.
For convenience, we denote \[
\mathcal{F}^s=\left\{ A\in \binom{[n]}{k}: A\cap [s]\neq \emptyset  \right\} {\rm \:\: and \:\:}
\mathcal{F}_s=\left\{ B\in \binom{[n]}{\ell}: [s] \subseteq B  \right\}.
\]
We know that
$|\mathcal F^s|=\binom {n}{k}-\binom {n-s}{k}$, $|\mathcal F_s|=\binom {n-s}{\ell-s}.$
From the assumptions (C1) and (C2) we have
$\mathcal F^r \subseteq \mathcal{A} \subseteq \mathcal F^{r+1}$ and
$\mathcal F_{r+1} \subseteq \mathcal{B} \subseteq \mathcal F_r.$
Clearly, $KG(\mathcal F^{r+1} \setminus \mathcal F^{r}, \mathcal F_r \setminus \mathcal F_{r+1})$ is regular and we have
\[
|\mathcal F^{r+1} \setminus \mathcal F^{r}|=\binom{n-r-1}{k-1}\:\: {\rm and}\:\:| \mathcal F_r \setminus \mathcal F_{r+1}|={\binom{n-r-1}{\ell - r}}.
 \]
Similarly, we may assume that $\cal A \setminus \mathcal F^r \neq \emptyset$.
To obtain $|\mathcal{A}||\mathcal{B}| < |\mathcal F^r||\mathcal F_r|$,
 by inequality (\ref{baseine}) it suffices to prove that
${\frac{|\mathcal{F}_{r}|}{|\mathcal{F}^{r}|}} < \frac{{\binom{n-r-1}{\ell - r}} }{\binom{n-r-1}{k-1}},$
which follows from
\[{\frac{|\mathcal{F}_{r}|}{|\mathcal{F}^{r}|}}
=\frac{\binom{n-r}{\ell-r}}{\sum_{1\leq i\leq r}\binom{n-i}{k-1}}
<\frac{\binom{n-r-1}{\ell-r}+\binom{n-r-1}{\ell-r}}{\binom{n-2}{k-1}+\binom{n-2}{k-1}}
<\frac{{\binom{n-r-1}{\ell - r}} }{\binom{n-r-1}{k-1}}.\]
This completes the proof. \qed

\subsection{Proof of Theorem \ref{theo3}}

At first, we introduce some recurrence relations and inequalities on binomial coefficients.

\begin{result}
Let $m$, $j$, $s$ be positive integers with $m\geq j+ s$. Then \begin{eqnarray}\label{beq2}\binom{m}{j+1}=\binom{m-1}{j}+\binom{m-2}{j}+\cdots+\binom{m-s}{j}+\binom{m-s}{j+1},\end{eqnarray}
and
\begin{eqnarray}\label{beq0}\binom{m}{j}=\binom{m-1}{j}+\binom{m-2}{j-1}+\cdots+\binom{m-s}{j-s+1}+\binom{m-s}{j-s}.\end{eqnarray}
In particular, \begin{eqnarray}\label{beq1}\binom{m}{j}=\binom{m-1}{j}+\binom{m-2}{j-1}+\cdots+\binom{m-j-1}{0}.\end{eqnarray}
\end{result}
\begin{result}
Let $n$,$k$,$\ell$,$s$ be positive integers with $n\geq k+\ell$ and $s\leq \ell$. Then \begin{eqnarray}\label{ineq24}\frac{\binom{n-2-s}{k-2}}{\binom{n-2-s}{\ell-s}}> \frac{\binom{n-3-s}{k-3}}{\binom{n-3-s}{\ell-s}}>\cdots> \frac{\binom{n-k-s}{0}}{\binom{n-k-s}{\ell-s}}\end{eqnarray}
and \begin{eqnarray}\label{eq35}\frac{\binom{n-2}{k-2}}{\binom{n-2}{\ell-1}}\leq\frac{\binom{n-3}{k-2}}{\binom{n-3}{\ell-2}}\leq\cdots\leq\frac{\binom{n-s}{k-2}}{\binom{n-s}{\ell-s+1}}.\end{eqnarray}
\end{result}

\begin{lemma}\label{lemmatheo2}
Let $n, k, \ell,r,s,t$  be positive integers with $\ell \ge k$, $n\geq k+\ell$, $1\leq r\leq k-1$, $0\leq t\leq k-r-1$,  and $2\leq s\leq \ell$.
Then
\begin{eqnarray}\label{ineq21}\frac{{\binom{n-r}{k-r}}-{\binom{n-r-s}{k-r}}}{{\sum_{1\leq i\leq r}\binom{n-i}{\ell-1}}+\sum_{0\leq j\leq t}{\binom{n-r-s-j}{\ell-s}}}
<\frac{{\binom{n-r-s-1}{k-r-1}}+{\binom{n-r-s-2}{k-r-2}}+\cdots+{\binom{n-r-s-t-1}{k-r-t-1}}}{{\binom{n-r-s-1}{\ell-s}}+{\binom{n-r-s-2}{\ell-s}}+\cdots+{\binom{n-r-s-t-1}{\ell-s}}}.\end{eqnarray}
\end{lemma}
\begin{proof}
First we consider the case $r=1$.
By (\ref{beq2}), (\ref{beq1}) and (\ref{ineq24}) we have
\begin{eqnarray}
\frac{\binom{n-1-s}{k-2}}{\binom{n-1-s}{\ell-s+1}}&=&\frac{\binom{n-2-s}{k-2}+\binom{n-3-s}{k-3}+\cdots+\binom{n-k-s}{0}}{
\binom{n-2-s}{\ell-s}+\binom{n-3-s}{\ell-s}+\cdots+\binom{n-k-s}{\ell-s}+\binom{n-k-s}{\ell-s+1}}\nonumber\\
&<&\frac{\binom{n-2-s}{k-2}+\cdots+\binom{n-k-s}{0}}{
\binom{n-2-s}{\ell-s}+\cdots+\binom{n-k-s}{\ell-s}}<\cdots \nonumber\\
&<&\frac{\binom{n-2-s}{k-2}+\cdots+\binom{n-t-2-s}{k-t-2}}{
\binom{n-2-s}{\ell-s}+\cdots+\binom{n-t-2-s}{\ell-s}}<\cdots\nonumber\\
&<&\frac{\binom{n-2-s}{k-2}}{
\binom{n-2-s}{\ell-s}}.\label{beq3}
\end{eqnarray}
Note that $(k+2)+(\ell-s+1)=k+\ell-s-1\leq n-s-1$.
If $k-2\ge \ell-s+1$, then  $\frac{\binom{n-1-s}{k-2}}{\binom{n-1-s}{\ell-s+1}}\geq 1>\frac{\binom{n-1}{k-1}}{\binom{n-1}{\ell-1}}$.  So in this case (\ref{ineq21}) can be obtained by (\ref{beq3}). Next we consider the case  $k-2\le \ell-s$. It is easy to verify that
$\frac{\binom{n-s}{k-2}}{\binom{n-s}{\ell-s+1}}
\leq\frac{\binom{n-s-1}{k-2}}{\binom{n-s-1}{\ell-s+1}}.$
By (\ref{beq2}) and (\ref{beq0}), we have
\[\frac{{\binom{n-1}{k-1}}-{\binom{n-1-s}{k-1}}}{{\binom{n-1}{\ell-1}}+\sum_{0\leq j\leq t}{\binom{n-s-j-1}{\ell-s}}}
=\frac{\sum_{i=2}^s{\binom{n-i}{k-2}}+\binom{n-s-1}{k-2}}{\sum_{i=2}^s\binom{n-i}{\ell-i+1}+\binom{n-s-1}{\ell-s}+\binom{n-s-1}{\ell-s-1}+\sum_{0\leq j\leq t}{\binom{n-s-j-1}{\ell-s}}}.\]
Applying ({\ref{eq35}}),
\begin{eqnarray*}
\label{ineq23}
\frac{\sum_{i=2}^s{\binom{n-i}{k-2}}}
{\sum_{i=2}^s\binom{n-i}{\ell-i+1}}
\leq \frac{\binom{n-s-1}{k-2}}{\binom{n-s-1}{\ell-s+1}}<\frac{\binom{n-2-s}{k-2}+\cdots+\binom{n-2-t-s}{k-t-2}}{
\binom{n-2-s}{\ell-s}+\cdots+\binom{n-2-t-s}{\ell-s}}.
\end{eqnarray*}
Then to complete the proof we only need to verify
the following inequality
$$\frac{\binom{n-s-1}{k-2}}{\binom{n-s-1}{\ell-s}+\binom{n-s-1}{\ell-s-1}+\sum_{0\leq j\leq t}{\binom{n-s-j-1}{\ell-s}}}<\frac{\binom{n-2-s}{k-2}+\cdots+\binom{n-2-t-s}{k-t-2}}{
\binom{n-2-s}{\ell-s}+\cdots+\binom{n-2-t-s}{\ell-s}}.$$
However, combining with (\ref{beq0}) and (\ref{ineq24}), we can derive
 \begin{small}
\begin{eqnarray*}
\frac{\binom{n-s-1}{k-2}}{\binom{n-s-1}{\ell-s}+\binom{n-s-1}{\ell-s-1}+\sum_{0\leq j\leq t}{\binom{n-s-j-1}{\ell-s}}}
&=&\frac{\binom{n-s-2}{k-2}+\cdots+\binom{n-s-t-1}{k-t-1}+\binom{n-s-t-2}{k-t-2}+\binom{n-s-t-2}{k-t-3}}{\binom{n-s-2}{\ell-s}+\cdots
+\binom{n-s-t-1}{\ell-s}+\binom{n-s-1}{\ell-s}+\binom{n-s-1}{\ell-s}+\binom{n-s-1}{\ell-s-1}}\\
&<&\frac{\binom{n-s-2}{k-2}+\cdots+\binom{n-s-t-1}{k-t-1}+\binom{n-s-t-2}{k-t-2}+\binom{n-s-t-2}{k-t-3}}{\binom{n-s-2}{\ell-s}+\cdots
+\binom{n-s-t-1}{\ell-s}+\binom{n-s-1}{\ell-s}+\binom{n-s-1}{\ell-s}} \\
&<&\frac{\binom{n-2-s}{k-2}+\cdots+\binom{n-2-t-s}{k-t-2}}{
\binom{n-2-s}{\ell-s}+\cdots+\binom{n-2-t-s}{\ell-s}}.
\end{eqnarray*}
\end{small}
Therefore, the result holds when $r=1$.

If $1<r\leq k-1$, set $n'=n-r+1$ and $k'=k-r+1$, then $n'\geq k'+\ell$, $0\le t \le k'-2$, $\ell>k'$ and
 \begin{footnotesize}
\begin{eqnarray*}
\frac{{\binom{n-r}{k-r}}-{\binom{n-r-s}{k-r}}}{{\sum_{1\leq i\leq r}\binom{n-i}{\ell-1}}+\sum_{0\leq j\leq t}{\binom{n-r-s-j}{\ell-s}}}
<\frac{{\binom{n'-1}{k'-1}}-{\binom{n'-1-s}{k'-1}}}{{\binom{n'-1}{\ell-1}}+\sum_{0\leq j\leq t}{\binom{n'-1-s-j}{\ell-s}}}<\frac{{\binom{n'-2-s}{k-r-1}}+{\binom{n'-3-s}{k-r-2}}+\cdots+{\binom{n'-t-2-s}{k-r-t-1}}}{{\binom{n'-2-s}{\ell-s}}+{\binom{n'-3-s}{\ell-s}}
+\cdots+{\binom{n'-t-2-s}{\ell-s}}}
\end{eqnarray*}
\end{footnotesize}
holds by the case $r=1$.
 Therefore we complete the proof.
\qed
\end{proof}
\bigskip
\newline
{\em Proof of Theorem \ref{theo3}.}
Let $\mathcal A\subseteq \binom{[n]}{k}$ and $\mathcal B\subseteq\binom{[n]}{\ell}$ be two maximal cross intersecting families with $|\mathcal B|>\binom{n-1}{\ell-1}$.
If $|\mathcal B|>\sum_{1\leq i\leq k}\binom{n-i}{\ell-1}$, then $\mathcal A=\emptyset$. So we assume
 \begin{small}
\begin{eqnarray*}
\sum_{1\leq i\leq r}\binom{n-i}{\ell-1}+\sum_{j=0}^{t-1} \binom{n-r-s-j}{\ell-s}< |\mathcal B|\leq \sum_{1\leq i\leq r}\binom{n-i}{\ell-1}+ \sum_{j=0}^{t} \binom{n-r-s-j}{\ell-s}
\end{eqnarray*}
\end{small}
for some $1\leq r\le k-1$, $1\le s\leq \ell$ and $1\leq t\leq k-r-1$.
Recall the definitions of $\mathcal{F}^{n,u,r,s}$, $\mathcal{F}_{n,u,r,s}$, $\mathcal{F}^{n,k,r,s,i}$ and $\mathcal{F}_{n,k,r,s,i}$ in Section 2.
Then $\mathcal A=\mathcal{F}_{n,u,r,s}\setminus(\mathcal{F}_{n,k,r,s,1}\cup\cdots\cup\mathcal{F}_{n,k,r,s,t-1}\cup N(\mathcal X_0))$ and $\mathcal B=\mathcal F^{n,u,r,s}\cup \mathcal{F}^{n,\ell,r,s-1,1}\cup\cdots\cup\mathcal{F}^{n,\ell,r,s-1,t-1}\cup\mathcal X_0$
for some $\mathcal X_0\subseteq \mathcal{F}^{n,\ell,r,s-1,t}$ and $N(\mathcal X_0)=\{A\in\mathcal{F}_{n,k,r,s,t}: \mbox{$A\cap B=\emptyset$ for some $B\in\mathcal X_0$}\}$. As the graph $KG(\mathcal{F}_{n,k,r,s,t},\mathcal{F}^{n,\ell,r,s-1,t})$  is regular,  $\frac{|N(\mathcal X)|}{|\mathcal X|}\geq \frac{|\mathcal{F}_{n,k,r,s,t}|}{|\mathcal{F}^{n,\ell,r,s-1,t}|}$ for all $\mathcal X\subseteq \mathcal{F}^{n,\ell,r,s-1,t}$.
Then, by Lemma \ref{lemmatheo2}, we have
\begin{eqnarray*}
\frac{|\mathcal{F}_{n,u,r,s}|}{|\mathcal B|}< \frac{\binom{n-r}{k-r}-\binom{n-r-s}{k-r}}{ \sum_{1\leq i\leq r}\binom{n-i}{\ell-1}+\sum_{j=0}^{t-1} \binom{n-r-s-j}{\ell-s}}<
\frac{{\binom{n-r-s-1}{k-r-1}}+{\binom{n-r-s-2}{k-r-2}}+\cdots+{\binom{n-r-s-t}{k-r-t}}}{{\binom{n-r-s-1}{\ell-s}}+{\binom{n-r-s-2}{\ell-s}}+\cdots+{\binom{n-r-s-t}{\ell-s}}}.
\end{eqnarray*}
However, $n\geq k+\ell$ implies $n-r-s-1\geq (k-r-1)+(\ell-s)$, and then  inequalities in (\ref{ineq24}) implies
\[\frac{\binom{n-r-s-1}{k-r-1}}{\binom{n-r-s-1}{\ell-s}}\geq\cdots\geq\frac{\binom{n-r-s-t+1}{k-r-t}}{\binom{n-r-s-t+1}{\ell-s}}\geq \frac{\binom{n-r-s-t+1}{k-r-t}}{\binom{n-r-s-t}{\ell-s}}.\]
Hence we can deduce
\begin{eqnarray*}
\frac{{\binom{n-r-s-1}{k-r-1}}+\cdots+{\binom{n-r-s-t}{k-r-t}}}{{\binom{n-r-s-1}{\ell-s}}+\cdots+{\binom{n-r-s-t}{\ell-s}}}\leq \frac{{\binom{n-r-s-1}{k-r-1}}+\cdots+\binom{n-r-s-t+1}{k-r-t+1}+|N(\mathcal X_0)|}{{\binom{n-r-s-1}{\ell-s}}+\cdots+{\binom{n-r-s-t+1}{\ell-s}}+|\mathcal X_0|}
\end{eqnarray*}
as $\frac{|N(\mathcal X_0)|}{|\mathcal X_0|}\geq \frac{|\mathcal{F}_{n,k,r,s,t}|}{|\mathcal{F}^{n,\ell,r,s-1,t}|}=\frac{\binom{n-r-s-t}{k-r-t}}{\binom{n-r-s-t}{\ell-s}}$ and $|\mathcal X_0|\leq \binom{n-r-s-t}{\ell-s}$.
Thus we obtain
$$\frac{|\mathcal{F}_{n,u,r,s}|}{|\mathcal B|}<\frac {|\mathcal{F}_{n,k,r,s,1}|+...+|\mathcal{F}_{n,k,r,s,t-1}|+|N(\mathcal X_0)|}{|\mathcal{F}^{n,\ell,r,s-1,1}|+...+|\mathcal{F}^{n,\ell,r,s-1,t-1}|+|\mathcal X_0|}= \frac{|\mathcal{F}_{n,u,r,s}\setminus \mathcal A|}{|\mathcal B\setminus \mathcal F^{n,u,r,s}|}= \frac{|N(\mathcal B\setminus \mathcal F^{n,u,r,s})|}{|\mathcal B\setminus \mathcal F^{n,u,r,s}|},$$
and hence
\[
|\mathcal{A}| |\mathcal{B}|\le |\mathcal{F}_{n,u,r,s}||\mathcal F^{n,u,r,s}| = \left(\binom{n-r}{k-r}-\binom{n-r-s}{k-r}\right) \left(\sum_{1\leq i\leq r}\binom{n-i}{\ell-1}+\binom{n-r-s}{\ell-s}\right).
\]
This completes the proof.
\qed

%

\section{The proof of Theorem \ref{maintheo} and preparations}

In this section, on the one hand, we use Lov\'{a}sz's theorem (Theorem \ref{Lov1979}) to complete the proof of the inequalities in the main theorem.
On the other hand, we use M\"ors' theorem (Theorem \ref{mors1}) to obtain the unique extremal families.
  \begin{proposition} \label{propfunction}
Let $n, k, \ell$  be positive integers with $n\ge  2\ell > 2k>0$.
If two families $\mathcal{A}\subseteq {\binom{[n]}{k}}$ and $\mathcal{B}\subseteq {\binom{[n]}{\ell}}$ are cross-intersecting with
$ {\binom{n-1}{k-1}}+\sum_{i=3}^4{\binom{n-i}{k-2}} \le |\mathcal{A}|\le {\binom{n-1}{k-1}}+\sum_{i=3}^5{\binom{n-i}{k-2}}$ with $n\ge \ell^2$, and $ {\binom{n-1}{k-1}}+\sum_{i=3}^{5}{\binom{n-i}{k-2}} \le |\mathcal{A}|\le {\binom{n-1}{k-1}}+\sum_{i=3}^{\ell+1}{\binom{n-i}{k-2}}$, then
\begin{small}
\begin{equation} \label{equamain}
|\mathcal{A}| |\mathcal{B}|\le \max \left\{ \left({\binom{n-1}{k-1}}+{\binom{n-2 }{k-1}}\right){\binom{n-2}{\ell-2}}, \left({\binom{n-1}{k-1}}+1\right)\left({\binom{n-1}{\ell-1}}-{\binom{n-k-1}{\ell-1}}\right)\right\}.
\end{equation}
\end{small}
 \end{proposition}

 \begin{proposition} \label{propextremal}
 The extremal families in Theorem \ref{maintheo} are unique up to isomorphism.
  \end{proposition}

We need the following notations and lemmas to prove Proposition \ref{propfunction}.
For fixed $n,k,\ell$, define a polynomial function of $x$ as
 \begin{small}
\begin{eqnarray}\label{varphix}
\varphi (x):= \left({\binom{n-1}{k-1}}+{\binom{n-2 }{k-1}}\right)\left({\binom{n-2}{\ell-2}}+{\binom{x}{n-\ell-1}}\right)-{\binom{n-2}{\ell-2}}{\binom{x}{k-1}}
\end{eqnarray}
 \end{small}
and denote
 \begin{small}
 \begin{eqnarray}\label{Gamma}
 \Gamma =\Gamma(n,k,\ell)=\max \left\{ \left({\binom{n-1}{k-1}}+{\binom{n-2 }{k-1}}\right){\binom{n-2}{\ell-2}}, \left({\binom{n-1}{k-1}}+1\right)\left({\binom{n-1}{\ell-1}}-{\binom{n-k-1}{\ell-1}}\right) \right\}.
 \end{eqnarray}
 \end{small}

Notice that $\varphi (n-5)=\left(\binom{n-1 }{ k-1}+\binom{n-2 }{ k-1}\right)\left(\binom{n-2}{ \ell-2}+\binom{n-5}{ \ell-4}\right)-\binom{n-5}{ k-1}\binom{n-2}{ \ell-2}$.
We first show that $\varphi (n-5)<\Gamma$,
which follows from the following two lemmas.

 \begin{lemma} \label{lemma(n-5)}
Let $n,k,\ell$ be positive integers with $n\ge 2\ell \ge 2k> 0$. Then $\varphi (n-5)<\Gamma$.
\end{lemma}

Since $\varphi (n-5)=\left(\binom{n-1 }{ k-1}+\binom{n-2 }{ k-1}\right)\left(\binom{n-2}{ \ell-2}+\binom{n-5}{ \ell-4}\right)-\binom{n-5}{ k-1}\binom{n-2}{ \ell-2}$, Lemma \ref{lemma(n-5)} can be deduced from the following two claims immediately.
 \begin{claim} \label{lemma3.2.5}
Let $n\ge 2\ell \ge 2k> 1.2\ell\ge 6$ be positive integers. Then
  \begin{small}
 \begin{equation} \label{eqklemma3.2.5}
 \left({\binom{n-1}{k-1}}+{\binom{n-2 }{k-1}}\right)\left({\binom{n-2}{\ell-2}}+{\binom{n-5}{\ell-4}}\right)<{\binom{n-1}{k-1}}\left({\binom{n-1}{\ell-1}}-{\binom{n-k-1}{\ell-1}}\right).
  \end{equation}
   \end{small}
\end{claim}

\begin{claim} \label{lemma3.2.6}
Let $n,k,\ell$ be positive integers satisfying $n\ge 2\ell \ge 2k$. If $\ell=4$ or $k\le 0.6\ell$, then
  \begin{equation} \label{eqklemma3.2.6}
  \left({\binom{n-1}{k-1}}+{\binom{n-2 }{k-1}}\right){\binom{n-5}{\ell-4}} <{\binom{n-5}{k-1}}{\binom{n-2}{\ell-2}}.
  \end{equation}
\end{claim}

The next lemma states that $ \varphi (n-4)<\Gamma $ if $n$ is large enough  with respect to $\ell$.
\begin{lemma} \label{lemmafn-4b}
Let $n, k, \ell$  be positive integers with $0<k < \ell <\ell^2 \le n$.
Then
$ \varphi (n-4)<\Gamma. $
\end{lemma}

Lemmas \ref{lemma(n-5)} and \ref{lemmafn-4b} imply that $\varphi (x)<\Gamma$ when $x$ is ``large'' with respect to $n$.
The next lemma shows that $\varphi (x)<\Gamma$ when $x$ is ``small'' with respect to $n$.
\begin{lemma} \label{lemmafn-l-1}
Let $n,k,\ell$ be positive integers with $n\ge \ell + 1$  and $\ell>k\ge 2$, then $ \varphi (n-\ell-1)<\Gamma$.
\end{lemma}

The proofs of \ref{lemma3.2.5}--\ref{lemma3.2.6} and Lemma \ref{lemmafn-4b}--\ref{lemmafn-l-1} are posted in Subsection 2,3,4, respectively.
Based on the results above, we are ready to prove our main result now.
\bigskip
\newline
{\em Proof of Theorem \ref{maintheo}}\:\:
Let $n, k, \ell$  be positive integers with $n\ge  2\ell >2k$ or $n>  2\ell=2k$.
Let $\mathcal{A}\subseteq {\binom{[n]}{k}}$ and $\mathcal{B}\subseteq {\binom{[n]}{\ell}}$ be cross-intersecting with $\cap_{F\in \mathcal{A} \cup \mathcal{B}} F=\emptyset$.
It is enough to consider $\cal{A}$ and $\cal{B}$ satisfying (C1) and (C2).
We first prove inequality (\ref{equamain}) by the following four statements.

\begin{itemize}
\item
If $|\mathcal{B}|> {\binom{n-1}{\ell-1}}$, then inequality (\ref{equamain}) holds by Theorem \ref{theo3} and the fact of
\begin{small}
\begin{eqnarray*}
\left({\binom{n-1}{\ell-1}}+1\right)\left({\binom{n-1}{k-1}}-{\binom{n-\ell-1}{k-1}}\right) \le \left({\binom{n-1}{k-1}}+1\right)\left({\binom{n-1}{\ell-1}}-{\binom{n-k-1}{\ell-1}}\right);
\end{eqnarray*}
\end{small}
\item
Otherwise if $|\mathcal{A}|\le {\binom{n-1}{k-1}}$ and $|\mathcal{B}|\le {\binom{n-1}{\ell-1}}$, then inequality (\ref{equamain}) holds by Theorem \ref{mors2}; 

\item Otherwise if ${\binom{n-1}{k-1}}+1 \le |\mathcal{A}|< {\binom{n-1}{k-1}}+\sum_{i=3}^4{\binom{n-i}{k-2}}$ or
$ {\binom{n-1}{k-1}}+\sum_{i=3}^4{\binom{n-i}{k-2}} \le |\mathcal{A}| < {\binom{n-1}{k-1}}+\sum_{i=3}^5{\binom{n-i}{k-2}}$ with $n<\ell^2$ or
$ |\mathcal{A}|\ge {\binom{n-1}{k-1}}+\sum_{i=3}^{\ell+1}{\binom{n-i}{k-2}}$, then inequality (\ref{equamain}) holds by Corollary \ref{coroA};

\item
  Otherwise ${\binom{n-1}{k-1}}+{\binom{n-3}{k-2}}+{\binom{n-4}{k-2}}\le |\mathcal{A}|< {\binom{n-1}{k-1}}+\sum_{i=3}^{5}{\binom{n-i}{k-2}}$ for $n\ge \ell^2$, or ${\binom{n-1}{k-1}}+\sum_{i=3}^{5}{\binom{n-i}{k-2}} \le |\mathcal{A}| < {\binom{n-1}{k-1}}+\sum_{i=3}^{\ell+1}{\binom{n-i}{k-2}}$,
then inequality (\ref{equamain}) holds by Proposition \ref{propfunction}.
\end{itemize}

Notice that from the proofs before, all equation in (\ref{equamain}) holds only if $|\cal{A}||\cal{B}|= \left({\binom{n-1}{k-1}}+{\binom{n-2 }{k-1}}\right){\binom{n-2}{\ell-2}}$ or $\left({\binom{n-1}{k-1}}+1\right)\left({\binom{n-1}{\ell-1}}-{\binom{n-k-1}{\ell-1}}\right)$, thus the extremal families in Theorem \ref{maintheo} are unique up to isomorphism by Proposition \ref{propextremal}.
This completes the proof.
\qed

\subsection{Proof of Claims \ref{lemma3.2.5}--\ref{lemma3.2.6} and Lemma \ref{lemmafn-4b}}
{\em Proof of Claim \ref{lemma3.2.5}}
\:\:
The condition $2k> 1.2\ell\ge 6$ implies that $k\ge 4$ and $\ell \ge 5$. 
Set \[f(n):=\frac{{\binom{n-5}{\ell-4}}}{{\binom{n-2}{\ell-2}}}={\frac{(n-\ell)(\ell-2)(\ell-3)}{(n-2)(n-3)(n-4)}}.\]
It is easy to see $f(n)\le f(2\ell)=\frac{\ell(\ell-3)}{4(\ell-1)(2\ell-3)}<\frac{(\ell-1)(\ell-2)}{4(\ell-1)(2\ell-3)}<\frac 1 8$.
However, we need a better upper bound when $k=4$, which implies that $\ell=5$ or $6$ from $2k> 1.2\ell\ge 6$.
By direct calculation, we know that $\frac{\ell(\ell-3)}{4(\ell-1)(2\ell-3)}<\frac 1 {10}$ for $\ell=5$ or $6$.
Therefore, combining with
${\binom{n-1}{k-1}}+{\binom{n-2 }{k-1}}=(2-{\frac{k-1}{n-1}}){\binom{n-1}{k-1}}$, we have
\[
\left({\binom{n-1}{k-1}}+{\binom{n-2 }{k-1}}\right)\left({\binom{n-2}{\ell-2}}+{\binom{n-5}{\ell-4}}\right)<
 \left\{
          \begin{array}{ll}
            {\frac{9}{8}}\left(2-{\frac{k-1}{n-1}}\right){\binom{n-1}{k-1}}{\binom{n-2}{\ell-2}}, & \hbox{if $\ell \ge 7$;} \\
            \frac {11} {10}\left(2-{\frac{k-1}{n-1}}\right){\binom{n-1}{k-1}}{\binom{n-2}{\ell-2}}, & \hbox{if $\ell \in \{5,6\}$.}
          \end{array}
        \right.
\]
We first consider the case $k\ge 5$. Notice that
 \begin{eqnarray}
  \frac {{\frac{9}{8}} \left(2-{\frac{k-1}{n-1}} \right)\binom{n-1}{k-1}{\binom{n-2}{\ell-2}}} {\binom{n-1}{k-1}\left({\binom{n-1}{\ell-1}}-{\binom{n-k-1}{\ell-1}}\right)}
 &=&  \frac {{\frac{9}{8}} \left(2-{\frac{k-1}{n-1}} \right){\binom{n-2}{\ell-2}}}  {{\binom{n-2}{\ell-2}}+{\binom{n-3}{\ell-2}}+\cdots+{\binom{n-k-1}{\ell-2}}} \nonumber\\
 &= & \frac { {\frac{9}{8}} \left(2-{\frac{k-1}{n-1}} \right) }{ 1+ {\frac{n-\ell}{n-2}}+ {\frac{(n-\ell)(n-\ell-1)}{(n-2)(n-3)}} + \cdots + {\frac{(n-\ell)\cdots(n-\ell-k+2)}{(n-2)\cdots(n-k)}} }   \nonumber\\
 &<& \frac { {\frac{9}{8}} \left(2-{\frac{0.6\ell-1}{n-1}} \right)}{ 1+ {\frac{n-\ell}{n-2}} + \cdots + {\frac{(n-\ell)\cdots(n-\ell-3)}{(n-2)\cdots(n-5)}}}, \label{lemma3.2.5a}
\end{eqnarray}
where the last inequality follows from $k\ge 5$.
Set \[g(n)=1+ {\frac{n-\ell}{n-2}} + \cdots + {\frac{(n-\ell)\cdots(n-\ell-3)}{(n-2)\cdots(n-5)}}-{\frac{9}{8}} \left(2-{\frac{0.6\ell-1}{n-1}} \right). \]
Notice that inequality (\ref{eqklemma3.2.5}) holds if RHS of inequality (\ref{lemma3.2.5a}) is larger than 1, that is, $g(n)>0$.
Now we will prove that $g(n)>0$.
Firstly,
\[g'(n)>{\frac{\ell-2}{(n-2)^2}}-{\frac{9}{8}}{\frac{0.6\ell-1}{(n-1)^2}}>{\frac{\ell-2-{\frac{9}{8}}(0.6\ell-1)}{(n-2)^2}}={\frac{0.325\ell-7/8 }{ (n-2)^2}}>0,\]
where the last inequality follows from $\ell \ge 5$.
Denote $h(\ell):=g(2\ell)$.
Secondly,
$$h'(\ell)=-\frac{3(624\ell^6 - 4880\ell^5 + 13672\ell^4 - 15820\ell^3 + 4559\ell^2 + 4230\ell - 2325)}{40(2\ell - 5)^2(\ell - 1)^2(2\ell - 1)^2(2\ell - 3)^2}.$$
Denote $m(\ell)=624\ell^6 - 4880\ell^5 + 13672\ell^4 - 15820\ell^3 + 4559\ell^2 + 4230\ell - 2325$.
Since $m(\ell)>(624\ell - 4880)\ell^5>0$ for $\ell\ge 8$, $m(7)=19045704$, $m(6)=5655435$  and $m(5)=1200300$,
we have $h'(\ell)<0$ for all $\ell \ge 5$.
Consequently,
$g(n)\ge g(2\ell)\ge \lim_{\ell\rightarrow +\infty} h(\ell)=0.025>0,$
and hence inequality (\ref{eqklemma3.2.5}) holds. Similarly, inequality (\ref{eqklemma3.2.5}) holds for $k=4$.
\qed
\medskip
\newline
{\em Proof of Claim \ref{lemma3.2.6}}
\:\:
By direct calculation, inequality (\ref{eqklemma3.2.6}) holds for $\ell=4$, and $k=2,3$, clearly.
So consider the case $k\le 0.6\ell$.
Set
\[f(n):=\begin{small}{\frac{\left({\binom{n-1}{k-1}}+{\binom{n-2 }{k-1}}\right){\binom{n-5}{\ell-4}}}{{\binom{n-5}{k-1}}{\binom{n-2}{\ell-2}}}}={\frac{(2n-k-1)(n-\ell)(\ell-2)(\ell-3)}{(n-k)(n-k-1)(n-k-2)(n-k-3)}}\end{small} \:\:{\rm and}\:\:
g(n)=:\text{log} f(n).\]
Firstly,
\[g'(n)={\frac{2}{2n-k-1}}+{\frac{1}{n-\ell}}-{\frac{1}{n-k}}-{\frac{1}{n-k-1}}-{\frac{1}{n-k-2}}-{\frac{1}{n-k-3}}<0\]
since ${\frac{2}{2n-k-1}}<{\frac{1}{n-k}}$ and ${\frac{1}{n-\ell}}<{\frac{3}{n-k-1}}<{\frac{1}{n-k-1}}+{\frac{1}{n-k-2}}+{\frac{1}{n-k-3}}$.
Hence $f(n)$ is decreasing for $n\ge 2\ell$, and we have
 \begin{equation} \label{eq3.2.6}
 f(n)\le f(2\ell) = {\frac{(4\ell-k-1)\ell(\ell-2)(\ell-3)}{(2\ell-k)(2\ell-k-1)(2\ell-k-2)(2\ell-k-3)}}.
 \end{equation}
Clearly, the RHS of (\ref{eq3.2.6}) is less than 1 for $k=2$, and $(k,\ell)=(3,6)$.
Notice that $k<0.6\ell$ with $\ell\le 6$ implies that $k\le 3$. 
So inequality (\ref{eqklemma3.2.6}) holds for $\ell\le 6$.

Now assume that $\ell\ge 7$.
Combining $k<0.6\ell$ and $\frac{\ell-i}{2\ell-k-i}<\frac{\ell}{2\ell-k}$ for $i=2,3$, we obtain that
 \[f(2\ell)= {\frac{2\ell-k}{2\ell-k-1}}{\frac{(4\ell-k-1)\ell(\ell-2)(\ell-3)}{(2\ell-k)^2(2\ell-k-2)(2\ell-k-3)}}  < {\frac{1.12(4\ell-k)\ell^3}{(2\ell-k)^4}}:={\frac{1.12(4-x)}{(2-x)^4}},\]
where $k=x\ell$ with $0<x\le 0.6$.
Since $\frac{1.12(4-x)}{(2-x)^4}$ is increasing on $(0,0.6]$, we have $\frac{1.12(4-x)}{(2-x)^4}\le {\frac{1.12\times3.4}{1.4^4}}<1$.
Hence inequality (\ref{eqklemma3.2.6}) holds.
\qed
\medskip
\newline {\em Proof of Lemma \ref{lemmafn-l-1}}\:\:
We first assume that $(k,\ell)\neq (2,3)$.
Notice that
\begin{small}
\begin{eqnarray*}
\varphi (n-\ell-1)&=&\left({\binom{n-1}{k-1}}+{\binom{n-2 }{k-1}}\right)\left({\binom{n-2}{\ell-2}}+1\right)-{\binom{n-2}{\ell-2}}{\binom{n-\ell-1}{k-1}} \\
&\le& \Gamma + {\binom{n-1}{k-1}}+{\binom{n-2 }{k-1}}-{\binom{n-2}{\ell-2}}{\binom{n-\ell-1}{k-1}}\\
&=& \Gamma + \left({\frac{n-1}{n-k}}+1\right){\binom{n-2 }{k-1}}-{\binom{n-2}{\ell-2}}{\binom{n-\ell-1}{k-1}}\\
&\le &  \Gamma,
\end{eqnarray*}
\end{small}
where the second inequality follows from
 $\left({\frac{n-1}{n-k}}+1\right){\binom{n-2 }{k-1}}<3 {\binom{n-2 }{k-1}}\le {\binom{n-2}{k-1} }{\binom{k+1}{k-1}} \le {\binom{n-2}{\ell-2}}{\binom{n-\ell-1}{k-1}}$.
 Now we assume that $(k,\ell)=(2,3)$. By direct calculation, we have
\[\Gamma=\max\{(2n-3)(n-2),n(2n-5)\}=2n^2-5n>2n^2-6n+5=\varphi (n-\ell-1).\]
\qed
\medskip
\newline
{\em Proof of Lemma \ref{lemmafn-4b}}\:\:
Notice that
\begin{small}
\begin{eqnarray*}
\varphi (n-4)&=&\hspace{-2mm}\left({\binom{n-1}{k-1}}+{\binom{n-2 }{k-1}}\right)\left({\binom{n-2}{\ell-2}}+{\binom{n-4}{\ell-3}}\right)-{\binom{n-2}{\ell-2}}{\binom{n-4}{k-1}} \\
&=&\hspace{-2mm}\left({\binom{n-1}{k-1}}+{\binom{n-2 }{k-1}}\right){\binom{n-2}{\ell-2}} +\left(\left({\binom{n-1}{k-1}}+{\binom{n-2 }{k-1}}\right){\binom{n-4}{\ell-3}}-{\binom{n-2}{\ell-2}}{\binom{n-4}{k-1}}\right) \\
&<& \hspace{-2mm}\Gamma +2{\binom{n-1}{k-1}}{\binom{n-4}{\ell-3}}-{\binom{n-2}{\ell-2}}{\binom{n-4}{k-1}}.
\end{eqnarray*}
\end{small}
To complete the proof, it suffices to prove that
\begin{eqnarray}\label{eqn-4l^2}
\frac{2{\binom{n-1}{k-1}}{\binom{n-4}{\ell-3}}}{{\binom{n-2}{\ell-2}}{\binom{n-4}{k-1}}}
&= &  \frac{(2\ell-4)(n-1)(n-\ell)}{(n-k)(n-k-1)(n-k-2)}\le 1.
\end{eqnarray}
In fact, for $n\ge \ell^2$,
$n-2>n-k$, $2\ell-4<n-k-2$ and $(n-k)+(n-k-1)+(n-k-2)=3n-3k-3\ge 2n+2\ell-7=(2\ell-4)+(n-1)+(n-2)$.
Thereby,  inequality (\ref{eqn-4l^2}) holds.
\qed
\medskip

\subsection{Proof of Proposition \ref{propfunction}}
Denote
\begin{small}
\begin{eqnarray*}
\mathcal{A}'=\left\{A\in {\binom{[3,n]}{k-1}}: \{2\}\cup A \in \mathcal{A}\right\} \: {\rm and} \:\:
\mathcal{B}'=\left\{B\in {\binom{[3,n]}{\ell-1}}: \{1\} \cup B \in \mathcal{B}\right\}.
\end{eqnarray*}
\end{small}
The cross-intersecting property of $\mathcal{A}$ and $\mathcal{B}$ yields that $\mathcal{A}'$ and $\mathcal{B}'$ are cross-intersecting with $|\mathcal{A}|={\binom{n-1}{k-1}}+|\mathcal{A}'|$ and $|\mathcal{B}|={\binom{n-2}{\ell-2}}+|\mathcal{B}'|$.
Recall that $\overline{\mathcal{B}'}=\{[3,n]\setminus B: B \in \mathcal{B}'\}\subseteq {\binom{[3,n]}{n-\ell-1}}$ with $|\overline{\mathcal{B}'}|=|\overline{\mathcal{B}}|$.
Denote  $|\overline{\mathcal{B}'}|={\binom{x}{n-\ell-1}}$, where $n-\ell-1\le x \le n-5$ for all possible $n$, and $n-5<x\le n-4$ for $n\le \ell^2$.
Notice that $\ell \ge n-x-1$ since $\mathcal{B}'\neq \emptyset$.
We first consider the case $n-\ell-1\le x \le n-5$.
By Theorem \ref{Lov1979}, $|\Delta_{k-1}(\overline{\mathcal{B}'})|\ge {\binom{x}{k-1}}$.
As $\mathcal{A}'$ and $\mathcal{B}'$ are cross-intersecting,  $A \nsubseteq B$ for all $A\in \mathcal{A}'$ and $B\in \overline{\mathcal{B}'}$,
which implies that
$
\mathcal{A}\cap \Delta_{k-1}(\overline{\mathcal{B}}) =\emptyset.
$
Hence
\begin{small}
\begin{eqnarray*}
|\mathcal{A}'| \le   {\binom{n-2}{k-1}}- |\Delta_{k-1}(\overline{\mathcal{B}'})|
\le  {\binom{n-2}{k-1}}-{\binom{x}{k-1}},
\end{eqnarray*}
\end{small}
and so
\begin{small}
\begin{eqnarray*}
|\mathcal{A}||\mathcal{B}| &\le& \left( {\binom{n-1}{k-1}}+{\binom{n-2}{k-1}}-{\binom{x}{k-1}}\right)\left( {\binom{n-2}{\ell-2}}+{\binom{x}{n-\ell-1}}\right)<\varphi (x).
\end{eqnarray*}
\end{small}
Recall the definitions of $\varphi (x)$ in (\ref{varphix}) and $\Gamma$ in (\ref{Gamma}).
To complete the proof of inequality (\ref{equamain}) in this case, it suffices to prove that $\varphi (x)< \Gamma$ for all $n-\ell-1\le x \le n-5$.
By Lemma \ref{lemma3.2.5} and Lemma \ref{lemma3.2.6}, we have $\varphi (n-5)<\Gamma$; and $\varphi (n-\ell-1)<\Gamma$ by Lemma \ref{lemmafn-l-1}.
We are going to prove that $\varphi (x)<\Gamma$ for all $x\in (n-\ell-1, n-5)$.
Notice that
 \begin{small}
\begin{eqnarray*}
 \varphi'' (x)&=& 
 \left({\binom{n-1}{k-1}}+{\binom{n-2 }{k-1}}\right){\binom{x}{n-\ell-1}}\left( {\frac{1}{x(x-1)}} + \cdots +{\frac{1}{(x-n+\ell+3)(x-n+\ell+2)}} \right) \\
  && -{\binom{n-2}{\ell-2}}{\binom{x}{k-1}}\left( {\frac{1}{x(x-1)}} + \cdots +{\frac{1}{(x-k+3)(x-k+2)}} \right).
 \end{eqnarray*}
 \end{small}
If $\varphi'' (x)> 0$ for all $x\in (n-\ell-1, n-5)$, then $\varphi (x)\le \max \{\varphi (n-5),\varphi (n-\ell-1)\}<\Gamma$.
So assume that $\varphi'' (x)\le 0$ for some $x\in (n-\ell-1, n-5)$.
Let $x_0 \in (n-\ell-1, n-5)$ be an arbitrary maximum element.
So $ \varphi'' (x_0) =0$.
As
\[ {\frac{1}{x_0(x_0-1)}} + \cdots +{\frac{1}{(x_0-n+\ell+3)(x_0-n+\ell+2)}} < {\frac{1}{x_0(x_0-1)}}  + \cdots +{\frac{1}{(x_0-k+3)(x_0-k+2)}},
\]
\newline
we have
 $({\binom{n-1}{k-1}}+{\binom{n-2 }{k-1}}){\binom{x_0}{n-\ell-1}}<{\binom{n-2}{\ell-2}}{\binom{x_0}{k-1}}$.
It follows that
\begin{small}
\begin{eqnarray*}
\varphi (x_0)&=&\left({\binom{n-1}{k-1}}+{\binom{n-2 }{k-1}}\right){\binom{n-2}{\ell-2}}+\left({\binom{n-1}{k-1}}+{\binom{n-2 }{k-1}}\right)\binom{x_0}{n-\ell-1}-{\binom{n-2}{\ell-2}}{\binom{x_0}{k-1}} \\
&<& \left({\binom{n-1}{k-1}}+{\binom{n-2 }{k-1}}\right){\binom{n-2}{\ell-2}}\le \Gamma.
\end{eqnarray*}
\end{small}
Similarly, we are done for $n-5\le x \le n-4$ with $n\ge \ell^2$, since $\varphi (n-4)<\Gamma$ by Lemma \ref{lemmafn-4b}.
\qed

\subsection{Proof of Proposition \ref{propextremal}}
Let $\mathcal{A}'\subseteq {\binom{[n]}{k}}$ and
$\mathcal{B}'\subseteq {\binom{[n]}{l}}$ be two cross-intersecting families with maximum product size.
 Firstly, we assume that $|\mathcal{A}'||\mathcal{B}'|=({\binom{n-1}{k-1}}+1)({\binom{n-1}{\ell-1}}-{\binom{n-k-1}{\ell-1}})$.
 Hence $|\mathcal{A}'|={\binom{n-1}{k-1}}+1$ and $|\mathcal{B}'|={\binom{n-1}{\ell-1}}-{\binom{n-k-1}{\ell-1}}$ by the before proofs.
Recall $\overline{\mathcal{B}'}=\{[n]\setminus B:B\in \mathcal{B}'\}$.
Then $\overline{\mathcal{B}'}\subseteq {\binom{[n]}{n-\ell}}$ with
\begin{small}\begin{equation}\label{equtheo1}
 |\overline{\mathcal{B}'}|={\binom{n-1}{\ell-1}}-{\binom{n-k-1}{\ell-1}}={\binom{n-2}{n-\ell}}+{\binom{n-3}{n-\ell-1}}+\dots+{\binom{n-k-1}{n-k-\ell+1}},
 \end{equation}\end{small}
which is the $(n-\ell)$-cascade of $|\overline{\mathcal{B}'}|$.
Notice that
\begin{small}\begin{eqnarray*}
\mathcal{C}^{(n-k)}( |\overline{\mathcal{A}'}|)&=&{\binom{[n-2]}{n-\ell}}\bigcup \left(\bigcup_{i=1}^{k-1}\left\{D\cup \{n-i,n-1\}: D\in {\binom{[n-i-2]}{n-\ell-i}}\right\}\right).
\end{eqnarray*}\end{small}
Therefore
\begin{small}\begin{eqnarray*}
|\Delta_{k}(\mathcal{C}^{(n-k)}(|\overline{\mathcal{B}'}|))|= {\binom{n-2}{k}}+{\binom{n-3}{k-1}}+\dots+{\binom{n-k-1}{1}}={\binom{n-1}{k}}-1.
\end{eqnarray*}\end{small}
By Theorem \ref{kkl}, we have
$|\Delta_{k}(\overline{\mathcal{B}'})|\ge |\Delta_{k}(\mathcal{C}^{(n-k)}(|\overline{\mathcal{B}'}|))|={\binom{n-1}{k}}-1.$
As $\mathcal{A}'$ and $\mathcal{B}'$ are cross-intersecting, $A' \nsubseteq \overline{B'}$ for all $\overline{B'} \in \overline{\mathcal{B}'}$ and $A'\in \mathcal{A}'$, which implies $\mathcal{A}'\cap \Delta_{k}(\overline{\mathcal{B}'}) =\emptyset$.
Hence
\begin{small}
\begin{eqnarray*}
|\mathcal{A}'|\le {\binom{n}{k}}-|\Delta_{k}(\overline{\mathcal{B}'})|\le {\binom{n}{k}}-{\binom{n-1}{k}}+1={\binom{n-1}{k-1}}+1=|\mathcal{A}'|.
\end{eqnarray*}
\end{small}
Thus all equalities in above must be hold, and hence $|\Delta_{k}(\overline{\mathcal{B}'})|= |\Delta_{k}(\mathcal{C}^{(n-k)}(|\overline{\mathcal{B}'}|))|$.
Notice that the RHS of (\ref{equtheo1}) is the $(n-\ell)$-cascade of $|\overline{\mathcal{A}'}|$.
Setting $\ell'=k$ and $s=k$.
Then by Theorem \ref{mors2},
 the solution of $|\Delta_{k}(\overline{\mathcal{B}'})|= |\Delta_{k}(\mathcal{C}^{(n-k)}(|\overline{\mathcal{B}'}|))|$ is unique (view $\overline{\mathcal{B}'}$ as a variable with fixed size).
 Thus $ \overline{\mathcal{B}'} \cong \mathcal{C}^{(n-\ell)}(|\overline{\mathcal{B}'}|) $.
 It is not hard to see that
 $\mathcal{A}'  \cong  \{A\in {\binom{[n]}{k}}: 1\in A\}\cup \{[2,\ell+1]\} $ and
  $ \mathcal{B} \cong \{B\in {\binom{[n]}{\ell}}: 1\in B, B\cap [2,\ell+1]\neq \emptyset\} $.

The similar way (more simpler) adapt to the another case, that is,
 $|\mathcal{A}'||\mathcal{B}'|=({\binom{n-1}{k-1}}+{\binom{n-2 }{k-1}}){\binom{n-2}{\ell-2}}$.
 This completes the proof.
\qed

\section*{Acknowledgements}
The authors thank Suijie Wang for many valuable suggestions to improve our presentation
of the paper.
The first author is  supported by the National Natural Science Foundation of Hunan Province, China (No. 2023JJ30385).
The second author is supported by the National Natural Science Foundation of China  (No.12371332 and
No.11971439).

\end{document}